\documentclass{amsart}
\usepackage[utf8]{inputenc}
\usepackage{amsmath}
\usepackage{amssymb}
\usepackage{amsfonts}

\usepackage{amsthm}

\usepackage{mathrsfs}
\usepackage{bbm}

\numberwithin{equation}{section}

\newcommand{\R}{\mathbb{R}}
\newcommand{\Z}{\mathbb{Z}}
\newcommand{\T}{\mathbb{T}}
\newcommand{\E}{\mathbb{E}}

\newcommand{\F}{\mathcal{F}}

\newcommand{\supp}{\mathrm{supp}\,}
\newcommand{\dist}{\mathrm{dist}}

\newcommand{\pphi}{\varphi}
\newcommand{\eps}{\varepsilon}

\providecommand{\dotdiv}{
	\mathbin{
		\vphantom{+}
		\text{
			\mathsurround=0pt 
			\ooalign{
				\noalign{\kern-.35ex}
				\hidewidth$\smash{\cdot}$\hidewidth\cr 
				\noalign{\kern.35ex}
				$-$\cr 
			}%
		}%
	}%
}

\makeatletter
\@namedef{subjclassname@2020}{\textup{2020} Mathematics Subject Classification}
\makeatother

\newtheorem{theorem}{Theorem}
\newtheorem{lemma}{Lemma}[section]

\begin{document}
	\title[LPR inequality for arbitrary Vilenkin systems]{Littlewood--Paley--Rubio de Francia inequality for unbounded Vilenkin systems}
	\author{Anton Tselishchev}
	\address{St. Petersburg Department, Steklov Math. Institute, Fontanka 27, St. Petersburg 191023 Russia}
	\email{celis-anton@yandex.ru}
	\keywords{Vilenkin systems, Fourier multipliers, Littlewood--Paley theory, Rubio de Francia inequality}
	\subjclass[2020]{42C10, 43A75}
	
	\begin{abstract}
		Rubio de Francia proved the one-sided version of Littlewood--Paley inequality for arbitrary intervals. In this paper, we prove the similar inequality in the context of arbitrary Vilenkin systems (that is, for functions on infinite products of cyclic groups). There are no assumptions on the orders of these groups.
	\end{abstract}

	\maketitle
	
	\section{Introduction}
	
	Let $f$ be a function on $\R$. For an arbitrary measurable set $E\subset\R$ we denote by $P_E$ the Fourier multiplier that acts on $f$ in a following way:
	$$
	P_E f = (\mathbbm{1}_E\widehat{f})^\vee.
	$$
	The famous Littlewood--Paley theorem (which was proved in \cite{L-P}; see also for example chapter 6 in \cite{Cl_Graf} or chapter 8 in \cite{MusSchlag}) states that if $I_j=[2^j, 2^{j+1}) \cup (-2^{j+1}, -2^j]$, then for every $p$, $1<p<\infty$, we have
	$$
	\|f\|_{L^p}\asymp \Big\| \Big( \sum_{j\in \Z} |P_{I_j} f|^2 \Big)^{1/2} \Big\|_{L^p}.
	$$
	The notation $A\asymp B$ means that there exist positive constants $c$ and $C$ such that $cA\leq B\leq CA$. We also use the notation $A\lesssim B$ to indicate that $A\leq CB$ for some  positive constant $C$. The constants in the Littlewood--Paley theorem depend on $p$ and of course do not depend on the function $f$.
	
	The sequence $\{2^j\}$ in the theorem may be replaced by any lacunary sequence $\{\lambda_j\}$, that is, $\lambda_j$ should satisfy the condition $\lambda_{j+1}/\lambda_j\ge c > 1$.
	
	In the paper \cite{RdF} Rubio de Francia proved that if $I_j$ are arbitrary disjoint intervals in $\R$ then for any $p$, $2\le p<\infty$, we have 
	\begin{equation}
		\label{RdF_orig}
		\Big\| \Big( \sum_{j} |P_{I_j} f|^2 \Big)^{1/2} \Big\|_{L^p} \lesssim  \|f\|_{L^p}.
	\end{equation}
	These results may be easily transferred to the case of functions on the torus $\T$. In this case, the intervals $I_j$ should be the subsets of $\Z$. By duality, Rubio de Francia's inequality may be reformulated in a following way: if $I_j$ are disjoint intervals in $\Z$ and $f_j$ are trigonometric polynomials such that $\supp \widehat{f}_j \subset I_j$, then there is an estimate
	$$
	\Big\| \sum_j f_j \Big\|_{L^p} \lesssim \Big\| \Big(\sum_j |f_j|^2\Big)^{1/2} \Big\|_{L^p}, \quad 1<p\le 2.
	$$
	In the article \cite{Bour} it was proved by Bourgain that this estimate also holds for $p=1$ and later in \cite{KisLP} Kislyakov and Parilov proved it also for all $p\in (0,2]$. For some other issues related to Rubio de Francia's inequality, see \cite{Lacey}.
	
	The aim of this paper is to prove the analogue of the original Rubio de Francia's inequality \eqref{RdF_orig} for Vilenkin systems instead of exponential functions. In order to formulate our main result, we need to introduce some definitions.
	
	Fix some sequence of positive integers $\{p_j\}_{j=0}^\infty$ such that $p_j\ge 2.$ Vilenkin system corresponding to this sequence is basically the set of characters on the group
	$$
	\prod_{j=0}^\infty \Z_{p_j}.
	$$
	This group can be identified with the segment $[0,1]$ (up to a countable number of points) in a natural way. This identification is given by the map
	$$
	\prod_{j=0}^\infty \Z_{p_j}\ni (a_0, a_1, \ldots)\mapsto \sum_{j=0}^\infty \frac{a_j}{\prod_{k=0}^j p_k}.
	$$
	Here we suppose that $0\le a_j\le p_j-1$.
	
	It will be more convenient for us to consider the functions on $[0,1]$. Each function in Vilenkin system is a product of generalized Rademacher functions which are constructed as follows. First, divide $[0,1]$ into $p_0$ equal segments. Then the function $r_0$ equals $e^{2\pi i k/p_0}$ on the $k$-th segment (we start enumeration of segments from $0$). Next, we divide each of the segments into $p_1$ equal segments and $r_1$ is a function which is equal to $e^{2\pi i k/p_1}$ on the $k$-th segment. By resuming this procedure, we get the sequence ${r_j}$ of generalized Rademacher functions.
	
	Let us introduce the following notation:
	$$
	m_k=\prod_{j=0}^{k-1} p_j, \quad k\ge 1.
	$$
	We also set $m_0=1$. Every number $n\in\Z_+$ has a unique representation in the mixed radix numerical system defined by $\{p_k\}$:
	\begin{equation}
		\label{mix_rad_rep}
		n=\sum_{j=0}^k n_jm_j, \quad 0\le n_j < p_j.
	\end{equation}
	In this case, we define the $n$-th Vilenkin function as
	$$
	w_n=\prod_{j=0}^k r_j^{n_j}.
	$$
	These functions were first introduced and studied in the paper \cite{Vil} by N. ~Vilenkin. 
	
	Obviously, if $p_j=2$ for all $j$, we get the Walsh system. In the general case, Vilenkin systems form an orthonormal basis in $L^2([0,1])$. Therefore, any function $f\in L^2$ has the following ``Vilenkin-Fourier'' representation:
	$$
	f=\sum_{n=0}^\infty a_n w_n.
	$$
	We will denote the coefficients $a_n$ by $\widehat{f}(n)$. Hence, for every function $f\in L^2$ we understand $\widehat{f}$ as a function on $\Z_+$ (and $\widehat{f}\in\ell^2(\Z_+)$). We have: 
	$$
	\widehat{f}(n)=\langle f, w_n\rangle =\int_0^1 f\bar{w}_n.
	$$ 
	On the other hand, for any function $g\in \ell^2(\Z_+)$ we can define its inverse Vilenkin-Fourier transform as the following function on $[0,1]$:
	$$
	g^\vee=\sum_{n=0}^\infty g(n)w_n.
	$$
	
	We are now ready to formulate the main result of this paper. From this moment, for any set $A\subset\Z_+$ we denote by $P_A$ the following ``Vilenkin-Fourier multiplier'':
	$$
	P_A f=(\mathbbm{1}_A\widehat{f})^\vee.
	$$
	
	\begin{theorem}
		\label{main_thm}
		Let $\{I_s\}$ be a sequence of disjoint intervals in $\Z_+$. Then for every $p\ge 2$ we have:
		$$
		\Big\| \Big( \sum_s |P_{I_s} f|^2 \Big)^{1/2} \Big\|_{L^p} \lesssim \|f\|_{L^p}.
		$$
		The constant in this inequality does not depend on the intervals $\{I_s\}$ nor on the sequence $\{p_j\}$.
	\end{theorem}
	
	For the case of Walsh systems (that is, when all $p_j$ are equal to $2$) this result was established by N. Osipov in the paper \cite{Osip}. Later, in \cite{Tsel} the argument from \cite{Osip} was generalized for bounded Vilenkin systems, which means that the sequence $\{p_j\}$ should be bounded by some constant.
	
	Several remarks are in order. First, we note that the proof of Rubio de Francia's inequality for the Walsh system relies on the Gundy's theorem for dyadic martingales. On the other hand, it is clear that if we have only one cyclic group then the proof of such inequality should invoke the methods from original Rubio de Francia's paper \cite{RdF} (that is, the methods of singular integral operators). Therefore, the proof of Theorem~1 should use both martingale and singular integral techniques.
	
	Next, we do not need an assumption of the boundedness of the sequence $\{p_j\}$. This boundedness is natural in many questions concerning Vilenkin systems since it implies the regularity of the underlying filtration (for details, see the next section). For example, if we consider the intervals $I_{k,j}=[jm_k, (j+1)m_k)$ (where $1\le j<p_k$) then it is proved in \cite{Wat} that for $p<2$ the estimate
	$$
	\Big\| \Big( \sum_{k=0}^\infty \sum_{j=1}^{p_k-1} |P_{I_{k,j}} f|^2 \Big)^{1/2} \Big\|_{L^p} \lesssim \|f\|_{L^p}.
	$$
	does not hold in general for arbitrary Vilenkin systems. It is worth noting here that this estimate for $p\ge 2$ (which is a special case of our Theorem~\ref{main_thm}) is known for arbitrary Vilenkin systems, see for example Theorem 2.11 in the book \cite{Weisz} where it is formulated as an embedding of different Hardy spaces.
	
	Even the boundedness of a single operator $P_I$ on $L^p$ for arbitrary interval $I\subset\Z_+$ is not completely trivial in our context of general Vilenkin systems. The fact that the norms of such operators are bounded by a constant that does not depend on an interval nor the sequence $\{p_j\}$ is proved in the paper \cite{Young_1} by Wo-Sang Young. We will use the methods of this paper together with their further development in \cite{Young_2} in the proof of our Theorem~\ref{main_thm}. Besides that, as we mentioned above, we will also need the ideas from Rubio de Francia's original paper \cite{RdF}. We also note here that in the paper \cite{Young_3} the Littlewood--Paley inequality (for intervals $[2^k, 2^{k+1})$) is proved in the context of Vilenkin systems.
	
	\section{Preliminaries}
	\subsection{Vilenkin systems and martingales}
	
	It is well known that Walsh system is strongly connected to dyadic martingales on $[0,1]$. In a similar way, Vilenkin system is also connected to a martingale. Let $\F_k$ be a sigma-algebra on $[0,1]$ generated by the segments
	$$
	[jm_{k}^{-1}, (j+1)m_k^{-1}), \quad 0\le j < m_k.
	$$
	We denote the collection of these segments by $\widetilde{\F}_k$. Then the function $r_k$ is constant on the segments from $\widetilde{\F}_{k+1}$. On the other hand, we see from the definition that the function $r_k$ has mean value zero over any segment from $\widetilde{\F}_k$. We will denote the expectation with respect to $\F_k$ as $\E_k$:
	$$
	\E_k(\cdot)=\E(\cdot|\F_k).
	$$
	It is easy to see that the operator $\E_k$ is connected to Vilenkin system in a following way:
	$$
	\E_k=P_{[0, m_k)}.
	$$
	In other words, we have the following identity:
	$$
	\E_k f = \sum_{j=0}^{m_k-1} \langle f, w_n\rangle w_n.
	$$
	This fact is well known and it is proved in exactly the same way as the similar fact for Walsh system, so we leave it without proof.
	
	We will use the following notation for martingale differences:
	$$
	\Delta_k=\E_{k+1}-\E_k, \quad k\ge 0.
	$$
	Obviously, we have
	$$
	\Delta_k = P_{\delta_k}, \ \text{where}\ \delta_k =[m_k, m_{k+1}).
	$$
	Each of the intervals $\delta_k$ is in turn the union of $p_k-1$ smaller intervals: if we denote
	$$
	\delta_{k,l}=[l m_k, (l+1) m_k), \quad 1\le l \le p_k-1,
	$$
	then we have
	$$
	\delta_k = \bigcup_{l=1}^{p_k - 1} \delta_{k,l}.
	$$
	We also introduce the notation
	$$
	\Delta_{k,l}=P_{\delta_{k,l}}.
	$$
	
	We will work a lot with the representations of integers in the mixed-radix numerical system as in the formula \eqref{mix_rad_rep}. It seems to be very illustrative to use the following notation in order to show that the integer $n$ is equal to the right hand side of the formula \eqref{mix_rad_rep}:
	$$
	n\sim \begin{pmatrix}
		m_0 & m_1 & \ldots & m_k\\
		n_0 & n_1 & \ldots & n_k
	\end{pmatrix}.
	$$
	We will also use similar notation for the intervals in $\Z_+$. For example, the interval $[0, \alpha m_k)$ can be written as follows:
	$$
	\begin{pmatrix}
		m_0 & m_1 & \ldots & m_{k-1} & m_k\\
		*   &  *  & \ldots &    *    & [0, \alpha)
	\end{pmatrix}.
	$$
	This notation seems to be self-explanatory.
	
	Suppose that we have two numbers $a$ and $b$ and let us write them using the notation we just introduced:
	$$
	a\sim \begin{pmatrix}
		m_0 & m_1 & \ldots & m_k\\
		\alpha_0 & \alpha_1 & \ldots & \alpha_k
	\end{pmatrix}, \ b\sim \begin{pmatrix}
		m_0 & m_1 & \ldots & m_k\\
		\beta_0 & \beta_1 & \ldots & \beta_k
	\end{pmatrix}.
	$$
	Then we see straight from the definition that 
	$$
	w_a w_b = w_{a\dotplus b},
	$$
	where
	$$
	a\dotplus b \sim \begin{pmatrix}
		m_0 & m_1 & \ldots & m_k\\
		(\alpha_0+\beta_0)\mathrm{mod}\, p_0 & (\alpha_1 + \beta_1)\mathrm{mod}\, p_1& \ldots & (\alpha_k + \beta_k)\mathrm{mod}\, p_k
	\end{pmatrix}.
	$$
	
	We also denote by $(\dotdiv n)$ the inverse of the number $n$ with respect to the operation~$\dotplus$. 
	
	\subsection{Combinatorial construction and scheme of proof}
	
	Recall that in the formulation of Theorem~\ref{main_thm} we have pairwise disjoint intervals $I_s=[a_s, b_s)$. Now we describe a decomposition of each interval into smaller subintervals. In order to simplify the notation we omit the index $s$ and describe a decomposition of an interval $[a, b)$. This combinatorial construction is presented in the paper \cite{Tsel} and it is a generalization of a similar construction for the case of Walsh systems from the paper \cite{Osip}.
	
	Suppose that
	$$
	a=\sum_{k=0}^N \alpha_k m_k, \quad b=\sum_{k=0}^N \beta_k m_k.
	$$
	First of all, we decompose the interval $[0,b)$. We write:
	$$
	[0, b) =\bigcup_{j=0}^N J_j,
	$$
	where $J_j$ is the following interval:
	$$
	J_j = \Big[\sum_{l=j+1}^N \beta_l m_l,\  \sum_{l=j}^N \beta_l m_l \Big).
	$$
	This interval can also be written as follows:
	$$
	J_j \sim \begin{pmatrix}
		m_0 & m_1 & \ldots & m_{j-1} & m_j & m_{j+1} & \ldots & m_N\\
		*   &   * & \ldots & *   & [0, \beta_j) & \beta_{j+1} & \ldots & \beta_N 
	\end{pmatrix}.
	$$
	If $\beta_j = 0$ then $J_j$ is empty.
	
	We have $a\in J_t$ for some $t$. It means that 
	$$
	\alpha_N=\beta_N,\ \alpha_{N-1}=\beta_{N-1}, \ldots, \alpha_{t+1}=\beta_{t+1}, \ \alpha_t < \beta_t.
	$$
	Therefore, now we need to decompose the interval
	$$
	[a, \beta_N m_N + \ldots + \beta_t m_t).
	$$
	We write it as
	$$
	\{a\}\cup \bigcup_{j=0}^t \widetilde{J}_j,
	$$
	where for $0\le j\le t-1$ we have
	$$
	\widetilde{J}_j \sim \begin{pmatrix}
		m_0 & m_1 & \ldots & m_{j-1} & m_j               & m_{j+1} & \ldots & m_N\\
		*   &   * & \ldots & *   & [\alpha_{j}+1, p_j-1] & \alpha_{j+1} & \ldots & \alpha_N 
	\end{pmatrix}.
	$$
	and
	$$
	\widetilde{J}_t \sim 
	\begin{pmatrix}
		m_0 & m_1 & \ldots & m_{t-1} & m_t               & m_{t+1} & \ldots & m_N\\
		*   &   * & \ldots & *   & [\alpha_{t}+1, \beta_t-1] & \alpha_{t+1} & \ldots & \alpha_N 
	\end{pmatrix}.
	$$
	Here $\widetilde{J}_j = \emptyset$ if $\alpha_j = p_j-1$ (for $0\le j\le t-1$) and $\widetilde{J}_t=\emptyset$ if $\alpha_t = \beta_t - 1$.
	
	Now we return the index $s$: for an interval $I_s= [a_s, b_s)$ we obtain its decomposition into intervals $J_{js}$ and $\widetilde{J}_{js}$. We are interested only in nonempty intervals: let $\Theta_s$ and $\widetilde{\Theta}_s$ be the sets of indices $j$ for which the intervals $J_{js}$ and respectfully $\widetilde{J}_{js}$ are nonempty. We finally have the following decomposition:
	$$
	I_s=\{a_s\}\cup \bigcup_{j\in\Theta_s} J_{js}\cup \bigcup_{j\in\widetilde{\Theta}_s} \widetilde{J}_{js}\cup \widetilde{J}_{t_s,s}.
	$$
	Here for some $s$ the interval $\widetilde{J}_{t_s,s}$ may be empty. We now have to prove the following four inequalities for $2 < p <\infty$ (note that Theorem~\ref{main_thm} is obvious for $p=2$ due to orthogonality of Vilenkin functions):
	\begin{gather}
		\Big\| \Big( \sum_s |P_{\{a_s\}}f|^2 \Big)^{1/2} \Big\|_{L^p}\lesssim \|f\|_{L^p};\label{main_first}\\
		\Big\| \Big( \sum_s \Big| \sum_{j\in\Theta_s} P_{J_{js}}f \Big|^2 \Big)^{1/2} \Big\|_{L^p}\lesssim \|f\|_{L^p};\label{main_second}\\
		\Big\| \Big( \sum_s \Big| \sum_{j\in\widetilde{\Theta}_s} P_{\widetilde{J}_{js}}f \Big|^2 \Big)^{1/2} \Big\|_{L^p}\lesssim \|f\|_{L^p};\label{main_third}\\
		\Big\| \Big(\sum_s |P_{\widetilde{J}_{t_s,s}f}|^2\Big)^{1/2}\Big\|_{L^p} \lesssim \|f\|_{L^p}\label{main_fourth}.
	\end{gather}
	The first estimate here is easy because of Parseval's theorem: 
	$$
	\Big( \sum_s |P_{\{a_s\}}f|^2 \Big)^{1/2} \leq \|f\|_{L^2}\leq \|f\|_{L^p}.
	$$
	The inequalities \eqref{main_second} and \eqref{main_third} are similar to each other and so we will only show how to prove \eqref{main_second}.
	
	The next two sections are devoted to the proofs of the estimates \eqref{main_second} and \eqref{main_fourth}. In the proof of \eqref{main_second} we will use duality together with the special version of Calderon-Zygmund decomposition which is presented in the papers \cite{Young_1} and \cite{Young_2} by Wo-Sang Young. 
	
	The estimate \eqref{main_fourth} is different (note for example that in the case of Walsh systems all intervals $\widetilde{J}_{t_s,s}$ are empty). In order to prove it, we will use the ideas from Rubio de Francia's original paper \cite{RdF} together with somewhat similar ideas from \cite{KisLP}. 
	
	\section{The proof of \eqref{main_second}}
	
	\subsection{Duality and application of Calderon--Zygmund decomposition}
	
	We now pass to the proof of the estimate \eqref{main_second}. Obviously, it is enough to prove our inequality for an arbitrary finite number of intervals (of course, the estimates should not depend on the number of intervals). As in the previous section, we will need the decomposition of $b_s$ in the mixed-radix numerical system:
	$$
	b_s=\sum_{k=0}^N \beta_{ks}m_k.
	$$ 
	Here $N$ is an arbitrary large number (since we have only a finite number of intervals, we may suppose that $b_s< m_{N+1}$ for some $N$).
	
	We see from the definition that
	$$
	J_{js}\dotdiv b_s=\bigcup_{k=p_j-\beta_{js}}^{p_j-1} \delta_{j,k}.
	$$
	Therefore, the operator $P_{J_{js}}$ can be written in a following way:
	$$
	P_{J_{js}} f = w_{b_s}(\Delta_{j, p_j-\beta_{js}}+\ldots + \Delta_{j, p_j-1})[w_{b_{s}}^{-1} f].
	$$ 
	Since $|w_{b_s}|=1$ everywhere on $[0,1]$, the estimate \eqref{main_second} is equivalent to the following inequality:
	$$
	\Big\| \Big( \sum_s \Big| \sum_{j\in\Theta_s} (\Delta_{j, p_j-\beta_{js}}+\ldots + \Delta_{j, p_j-1})[w_{b_s}^{-1}f] \Big|^2 \Big)^{1/2} \Big\|_{L^p} \lesssim \|f\|_{L^p}.
	$$
	Let us introduce the notation
	$$
	Q_{j,s}=\sum_{k=p_j-\beta_{js}}^{p_j-1} \Delta_{j,k}.
	$$
	Using this notation, we rewrite our desired inequality:
	$$
	\Big\| \Big( \sum_s \Big| \sum_{j\in\Theta_s} Q_{j,s}[w_{b_s}^{-1}f] \Big|^2 \Big)^{1/2} \Big\|_{L^p} \lesssim \|f\|_{L^p}.
	$$
	This inequality is equivalent to the fact that a certain operator is bounded from $L^p$ to $L^p(\ell^2)$. To be more precise, we consider the following operator $G$ that maps scalar-valued functions to vector-valued ones:
	$$
	Gf = \Big( \sum_{j\in\Theta_s}Q_{j,s}[w_{b_s}^{-1}f] \Big)_s.
	$$
	In order to prove that $G$ is a bounded operator from $L^p$ to $L^p(\ell^2)$, we use duality. An easy computation shows that for the dual operator $G^*$ we have
	$$
	G^* h = \sum_s w_{b_s}\sum_{j\in\Theta_s} Q_{j,s} h_s.
	$$
	Here $h=(h_s)_s$ is an $\ell^2$-valued function.
	
	We need to prove that $G^*$ is a bounded operator from $L^{p'}(\ell^2)$ to $L^{p'}$ (where $1/p + 1/p' = 1$). For $p=p'=2$ this fact easily follows from the Parseval's identity (we use here that the intervals $I_s$ and therefore $J_{js}$ are disjoint). By the Marcinkiewicz interpolation theorem, it is enough to show that $G$ is a bounded operator from $L^1(\ell^2)$ to $L^{1,\infty}$.
	
	In order to do this, we use the specific version of the Calderon--Zygmund decomposition which was first introduced in \cite{Young_1} and later modified for the $\ell^2$-valued case in \cite{Young_2}. We will need the latter modification, so we present the formulation of this Calderon--Zygmund type lemma here.
	
	\begin{lemma} \label{Cal-Zyg}
		Fix an arbitrary positive $\lambda>0$. Suppose that $h=(h_s)_s$ is an $\ell^2$-valued function such that $\|h\|_{L^1(\ell^2)}\leq \lambda$. Fix also the positive integers $\gamma_{ks}$ such that $0\le \gamma_{ks}<p_k$. Then there exist a collection of disjoint intervals $\mathcal{C}$ and $\ell^2$-valued functions $h^{(1)}$ and $h^{(2)}$ such that the following conditions hold:
		\begin{gather}
			h=h^{(1)}+h^{(2)};\\
			\|h^{(1)}\|_{L^\infty(\ell^2)}\leq C\lambda; \label{good_part_1}\\
			\|h^{(1)}\|_{L^1(\ell^2)}\leq C\|h\|_{L^1(\ell^2)};\label{good_part_2}\\
			\mathcal{C}=\bigcup_{k=0}^\infty \mathcal{C}_k,
		\end{gather}
		where each $\omega_j\in\mathcal{C}_k$ is $\mathcal{F}_{k+1}$-measurable and is strictly contained in one of the segments from $\widetilde{\F}_k$;
		\begin{gather}
			h^{(2)}(x)=0 \ \text{for}\ x\not\in \Omega=\bigcup_{\omega_j\in\mathcal{C}} \omega_j;\\
			\int_{\omega_j}h^{(2)}=0; \ \int_{\omega_j} h^{(2)}_s r_k^{\gamma_{ks}}=0 \ \text{for every}\ \omega_j\in\mathcal{C}_k;\label{cancellation_bad_part}\\
			\int_{\omega_j} |h^{(2)}|_{\ell^2}\leq C \int_{\omega_j} |h|_{\ell^2} \ \text{for every}\  \omega_j\in\mathcal{C};\label{bad_In_CZ}\\
			\sum_{\omega_j\in\mathcal{C}} |\omega_j| \leq \lambda^{-1}\|h\|_{L^1(\ell^2)} \label{lengths_of_omegas}.	\end{gather}
	\end{lemma}
	Here $h^{(1)}$ and $h^{(2)}$ are ``good" and ``bad" parts of the function $h$ respectively. The proof of this lemma is contained in the paper \cite{Young_2}.
	
	Our goal is to prove the weak type $(1,1)$ estimate for the operator $G^*$:
	$$
	|\{|G^*h|>\lambda\}|\leq\frac{C}{\lambda} \|h\|_{L^1(\ell^2)}.
	$$
	If $\|h\|_{L^1(\ell^2)}>\lambda$, then there is nothing to prove. Otherwise, we perform the Calderon--Zygmund type decomposition described in the lemma above (the choice of the numbers $\gamma_{ks}$ will be specified later). As usual, we write:
	$$
	|\{|G^*h|>\lambda\}| \leq |\{|G^*h^{(1)}|>\lambda/2\}| + |\{|G^*h^{(2)}| > \lambda/2\}|.
	$$ 
	The first summand here is easy to estimate, we simply use the $L^2$-boundedness of $G^*$ and write (again, as usual):
	$$
	|\{|G^* h^{(1)}| > \lambda/2\}|\leq \frac{4}{\lambda^2}\int |G^* h^{(1)}|^2 \lesssim \frac{1}{\lambda^2} \int |h^{(1)}|_{\ell^2}^2\lesssim \frac{1}{\lambda}\int |h^{(1)}|_{\ell^2}\lesssim \frac{1}{\lambda}\|h\|_{L^1(\ell^2)}.
	$$
	Here we used the properties \eqref{good_part_1} and \eqref{good_part_2}. Now we only need to prove the estimate
	\begin{equation}
		|\{|G^*h^{(2)}| > \lambda/2\}|\lesssim \frac{1}{\lambda}\|h\|_{L^1(\ell^2)}.
	\end{equation}
	
	\subsection{An estimate for a ``bad'' part: reduction to an inequality for one cyclic group} Recall that in Lemma~\ref{Cal-Zyg} we obtained a collection of intervals $\mathcal{C}$. For each $\omega_j\in\mathcal{C}$ we will need a slightly non-standard definition of $3\omega_j$ (which appeared in \cite{Young_1}). Suppose that $\omega_j\in \mathcal{C}_k$. It means that $\omega_j$ is $\F_{k+1}$-measurable and is contained in some $\widetilde{\omega}_j\in \widetilde{\F}_k$. Let us consider $\widetilde{\omega}_j$ as a circle and $\omega_j$ as an arc on this circle. Then $3\omega_j$ is an arc with the same center such that $|3\omega_j|=3|\omega_j|$. On our initial interval $[0,1]$ it means that $3\omega_j$ is also $\F_{k+1}$-measurable and is contained in $\widetilde{\omega}_j$ (but it is not necessarily a segment, it also can be the union of two segments). Using the property \eqref{lengths_of_omegas} we see that we only need to prove the following inequality:
	$$
	|\{x\not \in \bigcup 3\omega_j: |G^* h^{(2)}(x)|>\lambda/2\}| \lesssim \frac{1}{\lambda} \|h\|_{L^1(\ell^2)}.
	$$ 
	We start with the following simple estimates:
	\begin{multline*}
		|\{x\not \in \bigcup 3\omega_j: |G^* h^{(2)}(x)|>\lambda/2\}| \leq \frac{4}{\lambda} \int_{[0,1]\setminus \bigcup 3\omega_j} |G^* h^{(2)}| \\ \leq \frac{4}{\lambda} \sum_k \sum_{\omega\in\mathcal{C}_k} \int_{[0,1]\setminus 3\omega} |G^*h^{(\omega)}|,
	\end{multline*}
	where $h^{(\omega)}=h^{(2)}\mathbbm{1}_{\omega}$. Let us fix a non-negative integer $k$ and $\omega\in\mathcal{C}_k$. Now our goal is to prove the following estimate:
	\begin{equation}
		\label{to_prove_bad_part}
		\int_{[0,1]\setminus 3\omega} |G^* h^{(\omega)}|\lesssim \int_\omega |h^{(\omega)}|_{\ell^2}.
	\end{equation}
	This would be enough due to the property \eqref{bad_In_CZ} (the underlying constant in this inequality of course should not depend on $\omega$ nor on $k$).
	
	Recall the formula for $G^*$:
	$$
	G^* h^{(\omega)}=\sum_s w_{b_s} \sum_{l\in\Theta_s}Q_{l,s}h^{(\omega)}_s.
	$$
	In order to estimate this quantity, we will consider three cases.
	
	\subsubsection{Case 1: $l>k$} This case is easy due to the following simple observation.
	
	\begin{lemma}
		\label{orthogonality}
		Suppose that a function $f$ is such that $\supp f \subset e$ where $e$ is $\mathcal{F}_k$-measurable. Then $\supp \Delta_{k,j}f\subset e$ for all $1\le j\le p_k-1$.
	\end{lemma}
	\begin{proof}
		It is enough to show that if $f=0$ on some set $e_0\in\widetilde{\mathcal{F}}_k$, then $\Delta_{k,l}f=0$ on $e_0$. First we note that $\Delta_k f = 0$ on $e_0$ since this holds for both $\mathbb{E}_{k+1}$ and $\mathbb{E}_k$. Next, it follows straight from the construction of generalized Rademacher functions that $r_k^{j_1}$ and $r_k^{j_2}$ are orthogonal in $L^2(e_0)$ for $j_1\neq j_2$ while the functions $r_l$ are constants on $e_0$ for $l<k$. Therefore, $\Delta_{k,j_1} f$ and $\Delta_{k, j_2} f$ are orthogonal in $L^2(e_0)$ for $j_1\neq j_2$ and since we have
		$$
		0 = \Delta_k f = \sum_{j=1}^{p_k-1} \Delta_{k,j} f,
		$$ 
		it follows that $\Delta_{k,j} f=0$ on $e_0$ for each $j$.
	\end{proof}
	It is worth noting that this lemma is proved also in \cite{Tsel}. Now recall that
	$$
	Q_{l,s}=\Delta_{l,p_l-\beta_{ls}}+\ldots+\Delta_{l, p_l-1}.
	$$
	From the previous lemma it follows that $\supp Q_{l,s} h^{(\omega)}_s\subset \omega$ since $\supp h^{(\omega)}_s\subset \omega$ and $\omega$ is $\mathcal{F}_{k+1}$- and therefore $\mathcal{F}_l$-measurable. Hence the summands with $l>k$ do not affect the left-hand side of the formula \eqref{to_prove_bad_part}.
	
	\subsubsection{Case 2: $l<k$} In this case, we write:
	$$
	\Delta_{l,m}h^{(\omega)}_s=\sum_{n\in\delta_{l,m}}w_n \int h^{(\omega)}_s \bar{w}_n.
	$$
	All functions $r_j$ for $j<k$ are constants on intervals from $\widetilde{\mathcal{F}}_{k}$ and therefore are constants on $\omega$. Hence $w_n=\mathrm{const}$ on $\omega$ for $n\in\delta_{l,m}$ and we use the property \eqref{cancellation_bad_part} (the first part) to conclude that $Q_{l,s}h_s^{(\omega)}=0$ if $l<k$.
	
	\subsubsection{Case 3: $l=k$} The same argument as in the first case shows that $G^* h^{(\omega)}$ is supported in $\widetilde{\omega}$ --- the segment from $\widetilde{\F}_k$ which contains $\omega$. Therefore, we need to estimate the quantity
	\begin{multline}
		\label{before_simplif}
		\int_{\widetilde{\omega}\setminus 3\omega} \Big| \sum_{s:k\in\Theta_s} w_{b_s} Q_{k,s} h^{(\omega)}_s \Big|=\int_{\widetilde{\omega}\setminus 3\omega} \sqrt{\Big| \sum_{s:k\in\Theta_s} w_{b_s} Q_{k,s} h^{(\omega)}_s \Big|^2}\\=\int_{\widetilde{\omega}\setminus 3\omega} \sqrt{ \sum_{s: k\in\Theta_s} |Q_{k,s}h_s^{(\omega)}|^2 + g},
	\end{multline}
	where $g$ is a sum of functions of the form
	$$
	w_{b_{s_1}}\bar{w}_{b_{s_2}} Q_{k,s_1}h_{s_1}^{(\omega)} \overline{Q_{k,s_2}h^{(\omega)}_{s_2}}
	$$
	for $s_1\neq s_2$.
	
	If $e\in\widetilde{\F}_{k+1}$ then we see from the definition that $Q_{k,s}h_s^{\omega}$ is constant on $e$ (because for $n\in\delta_{k,j}$ we have $w_n=\mathrm{const}$ on $e$). On the other hand, for $s_1\neq s_2$ the functions $w_{b_{s_1}}$ and $w_{b_{s_2}}$ are orthogonal in $L^2(e)$. In order to see it, recall that an interval $J_{ks}$ has the following form:
	$$
	J_{ks} \sim \begin{pmatrix}
		m_0 & m_1 & \ldots & m_{k-1} & m_k & m_{k+1} & \ldots & m_N\\
		*   &   * & \ldots & *   & [0, \beta_{k,s}) & \beta_{k+1,s} & \ldots & \beta_{N,s} 
	\end{pmatrix}
	$$
	and the intervals $J_{ks_1}$ and $J_{ks_2}$ are disjoint. This fact implies then $\beta_{l,s_1}\neq \beta_{l,s_2}$ for some $l\ge k+1$. If we take the largest such $l$, then the orthogonality of $w_{b_{s_1}}$ and $w_{b_{s_2}}$ on $L^2(e_0)$ for $e_0\in\widetilde{\F}_l$ follows straight from the definition. Therefore, these functions are indeed orthogonal in $L^2(e)$ and it means that
	$$
	\int_e g = 0.
	$$
	
	The set $\widetilde{\omega}\setminus 3\omega$ is a union of several disjoint segments from $\widetilde{\F}_{k+1}$. We have seen that the function 
	$$
	\sum_{s: k\in\Theta_s} |Q_{k,s}h_s^{(\omega)}|^2
	$$
	is a constant on each of these segments. Let us take one of these segments $e$ and denote this constant by $c$. Then we use H\"older's inequality and write:
	\begin{multline*}
		\int_{e} \sqrt{ \sum_{s: k\in\Theta_s} |Q_{k,s}h_s^{(\omega)}|^2 + g} = \int_e \sqrt{c+g}\\ \leq \Big( \int_e c+g \Big)^{1/2} \cdot |e^{1/2}|=|e| c^{1/2} =\int_{e} \Big( \sum_{s: k\in\Theta_s} |Q_{k,s}h_s^{(\omega)}|^2\Big)^{1/2}.
	\end{multline*}
	Summing these inequalities over all $\widetilde{\F}_{k+1}\ni e\subset \widetilde{\omega}\setminus 3\omega$, we get that the quantity in the formula \eqref{before_simplif} does not exceed
	$$
	\int_{\widetilde{\omega}\setminus 3\omega} \Big( \sum_{s: k\in\Theta_s} |Q_{k,s}h_s^{(\omega)}|^2\Big)^{1/2}.
	$$
	Thus now we need to prove the following estimate:
	\begin{equation*}
		\int_{\widetilde{\omega}\setminus 3\omega} \Big( \sum_{s: k\in\Theta_s} |Q_{k,s}h_s^{(\omega)}|^2\Big)^{1/2}\lesssim \int_\omega \Big( \sum_{s: k\in\Theta_s} |h_s^{(\omega)}|^2 \Big)^{1/2}.
	\end{equation*}
	
	We will prove the following stronger inequality:
	\begin{equation}
		\label{est_bad_one_gr}
		\int_{\widetilde{\omega}\setminus 3\omega} \Big( \sum_{s: k\in\Theta_s} |Q_{k,s}h_s^{(\omega)}|^2\Big)^{1/2}\lesssim \Big( \sum_{s: k\in\Theta_s} \Big( \int_\omega |h_s^{(\omega)}| \Big)^2 \Big)^{1/2}.
	\end{equation}
	This estimate implies the previous one due to Minkowski integral inequality.
	
	Some informal remarks concerning the estimate \eqref{est_bad_one_gr} are in order. Note that now all our functions are supported inside $\widetilde{\omega}\in \widetilde{\F}_k$ and it is easy to see that both sides of our estimate depend only on the integrals of functions $h_s^{(\omega)}$ over segments from $\widetilde{\F}_{k+1}$. If we enumerate ``from left to right'' the segments from $\widetilde{\F}_{k+1}$ which lie inside $\widetilde{\omega}$:
	$$
	\widetilde{\omega}=e_0\cup e_1\cup\ldots\cup e_{p_k-1},
	$$
	we see that basically the estimate we have to prove is an inequality for the functions on the group $\Z_{p_k}=\{0, 1, \ldots, p_k-1\}$ and the operators $Q_{k,s}$ are Fourier multipliers on this group that correspond to the characteristic functions of the intervals. This remark will be more clear after the next subsection. Now we just note that this observation will also be useful to us later, in the proof of the inequality \eqref{main_fourth}.
	
	\subsection{An estimate for a ``bad'' part: final computations}
	
	Now we pass to the proof of the estimate \eqref{est_bad_one_gr}. By the definition, we have:
	$$
	\Delta_{k,l}h_s^{(\omega)}=\sum_{j\in\delta_{k,l}}w_j \int h_s^{(\omega)} \bar{w}_j.
	$$
	All functions $r_l$ with $0\le l\le k-1$ are constants on $\widetilde{\omega}$ and therefore this sum consists of $m_k$ equal summands. Hence we have:
	$$
	\Delta_{k,l}h_s^{(\omega)}=\Big(m_k r_k^l\int h_s^{(\omega)}r_k^{-l}\Big)\mathbbm{1}_{\widetilde{\omega}}.
	$$
	Now suppose that 
	$$
	\omega=e_m\cup e_{m+1} \cup \ldots \cup e_n, \quad 0\le m < n \le p_k-1.
	$$
	Then we have
	$$
	(\Delta_{k,l}h_s^{(\omega)})\vert_{e_t}=m_k e^{2\pi i l t/p_k}\sum_{j=m}^n \Big(\int_{e_j} h_s^{(\omega)}\Big) e^{-2\pi i l j/p_k}.
	$$
	Therefore, $Q_{k,s}h_s^{(\omega)}$ on $e_t$ is equal to
	$$
	m_k \sum_{l=p_k-\beta_{ks}}^{p_k-1} e^{2\pi i l t/p_k} \sum_{j=m}^n \Big(\int_{e_j} h_s^{(\omega)}\Big) e^{-2\pi i l j/p_k}.
	$$
	Now we change the order of summation and get:
	\begin{equation}
		\label{qks_1}
		(Q_{k,s}h_s^{(\omega)})|_{e_t}=m_k \sum_{j=m}^n\Big( \int_{e_j} h_{s}^{(\omega)} \Big) \sum_{l=p_k-\beta_{ks}}^{p_k-1} e^{2\pi i l(t-j)/p_k}.
	\end{equation}
	Now we note the following elementary identity which is not very difficult to check:
	$$
	\sum_{l=p-\alpha}^{p-1} e^{2\pi i l t/p}=\frac{1}{2}e^{-\frac{2\pi i \alpha t}{p}} - \frac{1}{2} + \frac{1}{2} i e^{-\frac{2\pi i \alpha t}{p}}\cot\frac{\pi t}{p} - \frac{1}{2}i\cot\frac{\pi t}{p}.
	$$
	This equality holds for every integer numbers $p\ge 2$, $\alpha < p$ and $p \nmid t$. We apply this identity to the right-hand side of the formula \eqref{qks_1} (note that we are interested in values of $Q_{k,s}h_s^{(\omega)}$ only outside $\omega$ and therefore $p_k\nmid (t-j)$) and get that
	\begin{multline*}
		(Q_{k,s}h_s^{(\omega)})|_{e_t}=m_k \sum_{j=m}^n\Big( \int_{e_j} h_{s}^{(\omega)} \Big) \Big( \frac{1}{2} e^{-\frac{2\pi i \beta_{ks}(t-j)}{p_k}}-\frac{1}{2}\\ + \frac{1}{2}i e^{-\frac{2\pi i \beta_{ks}(t-j)}{p_k}} \cot\frac{\pi(t-j)}{p_k}-\frac{1}{2} i \cot\frac{\pi(t-j)}{p_k}  \Big).
	\end{multline*}
	
	We now see that we had to take the numbers $\gamma_{ks}=\beta_{ks}$ in Lemma \ref{Cal-Zyg}. Then by \eqref{cancellation_bad_part} we have:
	\begin{gather}
		\label{cancellation_h}\sum_{j=m}^n \Big( \int_{e_j} h_s^{(\omega)} \Big) = 0;\\
		\label{cancellation_h_rad}\sum_{j=m}^n \Big( \int_{e_j} h_s^{(\omega)} \Big) e^{\frac{2\pi i \beta_{ks} j}{p_k}} = 0.
	\end{gather}
	Thus we can simplify the formula for $(Q_{k,s}h_s^{(\omega)})|_{e_t}$:
	$$
	(Q_{k,s}h_s^{(\omega)})|_{e_t}=\frac{i}{2}m_k \sum_{j=m}^n \Big( \int_{e_j}h_s^{(\omega)} \Big)\Big( e^{-\frac{2\pi i \beta_{ks}(t-j)}{p_k}} \cot\frac{\pi(t-j)}{p_k} - \cot\frac{\pi(t-j)}{p_k}  \Big).
	$$
	Let us introduce the following notation: for a function $g$ on $\omega$ and $e_t\not\subset\omega$ we set
	$$
	(\mathcal{H}_\omega(g))|_{e_t}=m_k \sum_{j=m}^n \Big( \int_{e_j} g \Big) \cot\frac{\pi(t-j)}{p_k}.
	$$
	Using this notation, we rewrite the formula for $Q_{k,s}h_s^{(\omega)}$ (outside $\omega$) once again:
	\begin{equation}
		\label{hilbert_transform}
		Q_{k,s}h_s^{(\omega)}=\frac{i}{2}\Big[ r_k^{-\beta_{ks}}\mathcal{H}_\omega(h_s^{(\omega)}r_k^{\beta_{ks}}) - \mathcal{H}_\omega(h_s^{(\omega)}) \Big].
	\end{equation}
	Now we will estimate the quantity
	$$
	\int_{\widetilde{\omega}\setminus 3\omega}\Big( \sum_{s:k\in\Theta_s} |\mathcal{H}_\omega (h_s^{(\omega)})|^2 \Big)^{1/2}.
	$$
	We will use only the fact that $\int h_s^{(\omega)}=0$ and since the same is true for the function $h_s^{(\omega)}r_k^{\beta_{ks}}$ (see \eqref{cancellation_h_rad}), the estimate of the quantity that appears from another summand in the formula \eqref{hilbert_transform} is similar.
	
	We use the formula \eqref{cancellation_h} and write:
	$$
	\mathcal{H}_\omega (h_s^{(\omega)})|_{e_t} = m_k \sum_{j=m}^n \Big( \int_{e_j} h_s^{(\omega)} \Big) \Big( \cot \frac{\pi(t-j)}{p_k} - \cot\frac{\pi (t-m)}{p_k} \Big).
	$$
	The fact that $e_t\not\subset 3\omega$ implies that for any real numbers $x_1, x_2 \in [m,n]$ we have
	\begin{equation*}
		\frac{\mathrm{dist}_{p_k} (x_1, t)}{p_k} \asymp \frac{\mathrm{dist}_{p_k} (x_2, t)}{p_k} \asymp \frac{\mathrm{dist}_{p_k} (m, t)}{p_k}.
	\end{equation*}
	Here by $\mathrm{dist}_p(a,b)$ we mean the following. We take $c$ such that $p | (c-(a-b))$ and $-p/2 < c\le p/2$. Then $\mathrm{dist}_p(a,b)=|c|$. In other words, for a positive integer $p$ the formula $\frac{\mathrm{dist}_p(a,b)}{p}$ can be understood as the distance between points $e^{2\pi i a/p}$ and $e^{2\pi i b/p}$ on the circle.
	
	Now we estimate the quantity $\mathcal{H}_\omega (h_s^{(\omega)})|_{e_t}$ for $e_t\not\subset 3\omega$ using the mean value theorem for the function $\cot x$:
	\begin{multline*}
		\big|\mathcal{H}_\omega (h_s^{(\omega)})|_{e_t}\big|\lesssim m_k \Big(\int_\omega |h_s^{(\omega)}|\Big) \cdot \frac{n-m}{p_k} \cdot \frac{p_k^2}{(\mathrm{dist}_{p_k}(t,m))^2}\\=\frac{m_{k+1}}{(\mathrm{dist}_{p_k}(t,m))^2}\Big(\int_\omega |h_s^{(\omega)}|\Big) \cdot (n-m).
	\end{multline*}
	Then we have:
	\begin{multline*}
		\int_{\widetilde{\omega}\setminus 3\omega}\Big( \sum_{s:k\in\Theta_s} |\mathcal{H}_\omega (h_s^{(\omega)})|^2 \Big)^{1/2}\\ \lesssim \Big(\sum_{t: e_t\not\subset 3\omega} \Big[ \sum_{s: k\in\Theta_s} \Big( \int_\omega |h_s^{(\omega)}| \Big)^2 \Big]^{1/2} \frac{1}{(\mathrm{dist}_{p_k}(t,m))^2}\Big) (n-m).
	\end{multline*}
	It is easy to see that
	$$
	\sum_{t:e_t\not\subset 3\omega} \frac{1}{(\mathrm{dist}_{p_k}(t,m))^2}\asymp \frac{1}{n-m}.
	$$
	Therefore we have the following estimate:
	$$
	\int_{\widetilde{\omega}\setminus 3\omega}\Big( \sum_{s:k\in\Theta_s} |\mathcal{H}_\omega (h_s^{(\omega)})|^2 \Big)^{1/2}\lesssim  \Big( \sum_{s: k\in\Theta_s} \Big( \int_\omega |h_s^{(\omega)}| \Big)^2 \Big)^{1/2}.
	$$
	It implies the inequality \eqref{est_bad_one_gr} and therefore the proof of	\eqref{main_second} is finished.
	
	It is now clear that the estimate that we proved in this subsection is basically an estimate for functions on just one cyclic group $\mathbb{Z}_{p_k}$. We will use the same scheme (the reduction to an inequality on one cyclic group) in the next section. However, we will use different method, namely, an estimate for the sharp maximal function instead of Calderon--Zygmund decomposition. Once this reduction is done we will repeat some arguments from Rubio de Francia's original paper \cite{RdF} (with certain necessary modifications).
	
	\section{The proof of \eqref{main_fourth}}
	
	\subsection{The definitions of generalized intervals, maximal function and Muckenhoupt weights}
	
	In this section we will use maximal functions and Muckenhoupt weights. Recall the definition of the standard Hardy--Littlewood maximal function (say, on the real line):
	$$
	M_q f(x)=\sup_{I\ni x} \Big( \frac{1}{|I|}\int_I |f(x)|^q\, dx\Big)^{1/q}
	$$
	where the supremum is taken over all intervals which contain the point $x$. However, we should slightly modify this definition in order to make it more convenient for our needs. The following definition is taken from \cite{Young_weights}.
	
	We denote by $\mathscr{F}_k$ the collection of the following subsets of $[0,1]$: $\omega\in\mathscr{F}_k$ if $\omega$ is a proper subset of an element $\widetilde{\omega}\in\widetilde{\F}_k$ and if we consider $\widetilde{\omega}$ as a circle then $\omega$ is a union of some consecutive intervals from $\widetilde{\F}_{k+1}$ (in other words, using notation from the end of Subsection 3.2 we treat the intervals $e_0$ and $e_{p_k-1}$ as adjacent, so for example $e_0\cup e_{p_k-1}$ is a generalized interval). We denote by $\mathscr{F}$ the union $\bigcup_{k\ge 0} \mathscr{F}_k$ and the elements of $\mathscr{F}$ are called the generalized intervals.
	
	Now the modification of the definition of maximal function is straightforward: we should use the generalized intervals instead of the usual ones. That is, from this moment we use the following definition:
	$$
	M_q f(x)=\sup_{\mathscr{F}\ni\omega\ni x} \Big( \frac{1}{|\omega|}\int_\omega |f(x)|^q\, dx\Big)^{1/q}.
	$$
	
	The definition of Muckenhoupt weights now takes the following form: a non-negative function $w$ is in the class $A_p$ for $1<p<\infty$ if
	$$
	\sup_{\omega\in\mathscr{F}} \Big( \frac{1}{|\omega|} \int_\omega w(x)\, dx \Big) \Big( \frac{1}{|\omega|} \int_{\omega} w(x)^{1-p'}\, dx \Big)^{p-1}=[w]_{A_p}<\infty.
	$$
	
	Such maximal functions and Muckenhoupt weights share all usual properties of their ``standard'' counterparts; in particular, the operator $M_q$ is bounded on $L^p(w)$ for $w\in A_{p/q}$ for $p>q$. We refer the reader to the book \cite{GarRubio} for the theory of Muckenhoupt weights. The standard properties of these classes of weights in our context (including the boundedness of the operator $M_q$) are presented in the paper \cite{Young_weights}.

	\subsection{A change of notation}
	
	We now pass to the proof of the inequality
	$$
	\Big\| \Big(\sum_s |P_{\widetilde{J}_{t_s,s}}f|^2\Big)^{1/2}\Big\|_{L^p} \lesssim \|f\|_{L^p},
	$$
	where $\widetilde{J}_{t_s,s}$ are pairwise disjoint intervals of the following form:
	$$
	\widetilde{J}_{t_s,s}\sim 
	\begin{pmatrix}
		m_0 & m_1 & \ldots & m_{t_s-1} & m_{t_s}               & m_{t_s+1} & \ldots & m_N\\
		*   &   * & \ldots & *   & [\alpha_{t_s}+1, \beta_{t_s}-1] & \alpha_{t_s+1} & \ldots & \alpha_N 
	\end{pmatrix}.
	$$
	Let us slightly change the notation. For a number $t$, $0 \le t < N$, we denote by $\Omega_t$ the set of numbers of the form
	$$
	\varkappa=\varkappa_N m_N + \ldots + \varkappa_{t+1} m_{t+1}, \quad 0\le \varkappa_j <p_j.
	$$
	For a fixed $t$ we have a collection of intervals of the following form:
	\begin{equation}
		\label{int_first}
		I_{t,s}^\varkappa \sim
		\begin{pmatrix}
			m_0 & m_1 & \ldots & m_{t-1} & m_{t}               & m_{t+1} & \ldots & m_N\\
			*   &   * & \ldots & *   & [\alpha_{s}, \beta_{s}] & \varkappa_{t+1} & \ldots & \varkappa_N 
		\end{pmatrix}
	\end{equation}
	for different values of $\varkappa\in\Omega_t$ (of course it is possible that no intervals correspond to some values of $t$ and $\varkappa\in\Omega_t$). Here $\alpha_s$ and $\beta_s$ are simply some positive integers such that all these intervals are pairwise disjoint (we do not need their connection to the original numbers $a_s$ and $b_s$). Using the new notation, the desired inequality can be rewritten in a following form:
	\begin{equation}
		\label{main_rewritten}
		\Big\| \Big( \sum_t \sum_{\varkappa\in\Omega_t} \sum_s |P_{I_{t,s}^\varkappa} f|^2 \Big)^{1/2} \Big\|_{L^p} \lesssim \|f\|_{L^p}.
	\end{equation}
	We note that for a fixed $t$ and a fixed $\varkappa\in \Omega_t$ this is equivalent to the original Rubio de Francia's inequality on the one cyclic group $\mathbb{Z}_{p_t}$ (in the same sense as in the subsection 3.3). So, we need to follow the route from the paper \cite{RdF}: first we should refine our intervals in order to get the ``well-distributed family'' and then consider the smooth version of the square function. The fact that we are dealing with the cyclic groups instead of $\R$ makes the notation and computations rather bulky.
	
	\subsection{Refinement of the intervals: reduction to a well-distributed family}
	
	At first we recall the definition of Whitney decomposition of an interval $[a,b]\subset\R$ which is the same as in \cite{RdF}. For the interval $[0,1]$ the collection $W([0,1])$ consists of the following intervals:
	$$
	\Big\{[2^{-(k+1)}/3, 2^{-k}/3] \Big\}_{k=0}^\infty; \ [1/3, 2/3]; \ \Big\{ [1-2^{-k}/3, 1-2^{-(k+1)}/3] \Big\}_{k=0}^\infty.
	$$
	For an arbitrary interval $[a,b]$ we simply transfer this decomposition using the affine mapping between $[0,1]$ and $[a,b]$ and get the collection $W([a,b])$. The advantage of such refinement is that we have $3I\subset [a,b]$ if $I\in W([a,b])$ and $\sum_{J\in W([a,b])} \mathbbm{1}_{2J}(x)\leq 5$ for all $x$.
	
	Now, for each interval $I_{t,s}^\varkappa$ given by the formula \eqref{int_first} we consider the interval
	$$
	\widetilde{I}_{t,s}^\varkappa=\Big(\frac{\alpha_s-1/2}{p_t}, \frac{\beta_s+1/2}{p_t} \Big) \subset [0,1].
	$$
	Note that for fixed numbers $t$ and $\varkappa$ the intervals $\widetilde{I}_{t,s}^\varkappa$ are pairwise disjoint. 
	
	We consider the Whitney decomposition of these intervals. After that, for each interval $\widetilde{I}_{t,s}^{\varkappa, k}\in W(\widetilde{I}_{t,s}^\varkappa)$ we ``transfer it back'' to $\mathbb{Z}_+$ in a clear way: $I_{t,s}^{\varkappa,k}$ consists of all points of the form
	$$
	\begin{pmatrix}
		m_0 & m_1 & \ldots & m_{t-1} & m_{t}               & m_{t+1} & \ldots & m_N\\
		*   &   * & \ldots & *   & j & \varkappa_{t+1} & \ldots & \varkappa_N 
	\end{pmatrix}
	$$
	where $j/p_t\in \widetilde{I}_{t,s}^{\varkappa, k}$. Of course, in this way we get only a finite number of non-empty subintervals of $I_{t,s}^\varkappa$.
	
	We apply the procedure described above to each interval $I_{t,s}^\varkappa$ and get in this way its decomposition:
	$$
	I_{t,s}^\varkappa=I_{t,s}^{\varkappa,1} \cup \ldots \cup I_{t,s}^{\varkappa, k_s}.
	$$
	The following lemma implies that it is enough to estimate the square function with respect to this refined collection of intervals.
	
	\begin{lemma}
		\label{refinement}
		Consider the following functions:
		$$
		Q_{I_{t,s}^\varkappa} f = \Big( \sum_{k=1}^{k_s} |P_{I_{t,s}^{\varkappa,k}} f|^2 \Big)^{1/2}.
		$$ 
		Then the following holds for $1<p<\infty$:
		$$
		\Big\| \Big( \sum_t \sum_{\varkappa\in\Omega_t} \sum_s |P_{I_{t,s}^{\varkappa}} f|^2 \Big)^{1/2} \Big\|_{L^p} \asymp \Big\| \Big( \sum_t \sum_{\varkappa\in\Omega_t} \sum_s |Q_{I_{t,s}^{\varkappa}} f|^2 \Big)^{1/2} \Big\|_{L^p}.
		$$
	\end{lemma}
	
	\begin{proof}[Sketch of proof]
		This lemma is much like \cite[Lemma 2.3]{RdF}. The main idea is to prove the following inequality  for arbitrary functions $g_{t,s}^\varkappa$:
		\begin{equation}
			\label{discr_lp}
			\Big\| \Big( \sum_t \sum_{\varkappa\in\Omega_t} \sum_s |Q_{I_{t,s}^{\varkappa}} g_{t,s}^\varkappa|^2 \Big)^{1/2} \Big\|_{L^p}\lesssim \Big\| \Big( \sum_t \sum_{\varkappa\in\Omega_t} \sum_s |g_{t,s}^\varkappa|^2 \Big)^{1/2} \Big\|_{L^p}.
		\end{equation}
		Substituting $g_{t,s}^\varkappa=P_{I_{t,s}^\varkappa} f$ gives us the inequality ``$\gtrsim$'' in Lemma and the reverse inequality follows by duality.
		
		The estimate \eqref{discr_lp} would follow from the fact that the operators $Q_{I_{t,s}^\varkappa}$ are uniformly bounded in $L^2(w)$ for $w\in A_2$ (this connection between vector-valued and weighted norm inequalities is explained for example in \cite[Section V.6]{GarRubio}; besides that, an elementary proof of the extrapolation theorem for $A_p$ weights and  connection to vector-valued inequalities which works without changes in our context is presented in \cite[Chapters 2--3]{CrMaPe}). Thus, we need to prove the following estimate:
		$$
		\Big\| \Big( \sum_{k=1}^{k_s} |P_{I_{t,s}^{\varkappa,k}} g|^2 \Big)^{1/2} \Big\|_{L^2(w)}\leq C([w]_{A_2}) \|g\|_{L^2(w)}.
		$$
		Each operator $P_{I_{t,s}^{\varkappa, k}}$ can be written as follows:
		$$
		P_{I_{t,s}^{\varkappa, k}} g = w_{\varkappa}P_{I_{t,s}^{0,k}}[w_{\varkappa}^{-1} g].
		$$
		Therefore, we need to prove the following:
		\begin{equation}
			\label{weighted_lp_1}
			\Big\| \Big( \sum_{k=1}^{k_s} |P_{I_{t,s}^{0,k}} h|^2 \Big)^{1/2} \Big\|_{L^2(w)}\leq C([w]_{A_2}) \|h\|_{L^2(w)},
		\end{equation}
		and then apply this inequality to the function $h=w_\varkappa^{-1} g$. The intervals $I_{t,s}^{0,k}$ have the following form:
		$$
		I_{t,s}^{0, k} \sim
		\begin{pmatrix}
			m_0 & m_1 & \ldots & m_{t-1} & m_{t}               & m_{t+1} & \ldots & m_N\\
			*   &   * & \ldots & *   & [\alpha_{s}^{(k)}, \beta_{s}^{(k)}] & 0 & \ldots & 0 
		\end{pmatrix}.
		$$
		It means that the operator $P_{I_{t,s}^{0,k}}$ is a sum of certain operators $\Delta_{t,j}$:
		$$
		P_{I_{t,s}^{0,k}}=\sum_{j=\alpha_s^{(k)}}^{\beta_s^{(k)}} \Delta_{t,j}.
		$$
		Thus if we change $h$ to $\E_{t+1} h$, it does not affect the left-hand side of the inequality \eqref{weighted_lp_1}. Besides that, we have the following estimate:
		$$
		\|\E_{t+1} h\|_{L^2(w)} \leq C([w]_{A_2}) \|h\|_{L^2(w)}.
		$$
		(This inequality is easy to prove on each $e\in\widetilde{\F}_{t+1}$ simply using the definition of the quantity $[w]_{A_2}$.)
		
		It means that in the estimate \eqref{weighted_lp_1} we can assume without loss of generality that $h$ is constant on the segments from $\widetilde{\F}_{t+1}$. 
		
		Our goal is to prove the following inequality for any $\widetilde{\omega}\in\widetilde{\F}_t$:
		\begin{equation}
			\label{lp_cyclic_weights}
			\int_{\widetilde{\omega}}  \Big( \sum_{k=1}^{k_s} |P_{I_{t,s}^{0,k}} h|^2 \Big) w \leq C([w]_{A_2}) \int_{\widetilde{\omega}} |h^2|w.
		\end{equation}
		Suppose that $\widetilde{\omega}$ is a union of the following consecutive segments from $\widetilde{\F}_{t+1}$:
		\begin{equation}
			\label{enum_int}
			\widetilde{\omega}=e_0\cup e_1\cup\ldots\cup e_{p_t-1}.
		\end{equation}
		Then we introduce the weight $v$ on $\mathbb{Z}_{p_t}=\{0, 1, \ldots, p_t-1\}$:
		$$
		v(m)=\int_{e_m} w.
		$$
		It is easy to see that $v$ is an $A_2$ weight on $\mathbb{Z}_{p_t}$ and therefore it enough to prove the inequality \eqref{lp_cyclic_weights} with $w$ replaced by the weight that equals $v(m)$ on $e_m$; since now all our functions in this inequality are constants on the segments from $\widetilde{\F}_{t+1}$, it is easy to see that it is equivalent to a weighted version of the Littlewood--Paley inequality on one cyclic group $\mathbb{Z}_{p_t}$. We postpone the proof of this inequality till the next section; for a moment we just note that such inequality in case of functions on $\R$ is proved in \cite{Kurtz}.
	\end{proof}
	
	We apply this lemma in order to refine each interval in our collection. Now we need to prove the estimate
	\begin{equation}
		\label{main_refined}
		\Big\| \Big( \sum_t \sum_{\varkappa\in\Omega_t} \sum_s \sum_{k=1}^{k_s} |P_{I_{t,s}^{\varkappa,k}} f|^2 \Big)^{1/2} \Big\|_{L^p} \lesssim \|f\|_{L^p}.
	\end{equation}
	
	\subsection{The ``smooth'' operator}
	
	We continue to follow the scheme of proof from the paper \cite{RdF}. The next step is to consider the smooth version of the square function.
	
	Let $\phi \in C_0^\infty (\mathbb{R})$ be an arbitrary function such that $\phi\equiv 1$ on $[-2, 2]$, $\phi \equiv 0$ outside the interval $[-2-1/100, 2+1/100]$ and $0\le \phi(x)\le 1$ for all $x$.
	
	We divide each interval $I_{t,s}^{\varkappa,k}$ into 7 consecutive intervals of almost equal lengths: this is done by dividing each interval $\widetilde{I}_{t,s}^{\varkappa,k}\subset\mathbb{R}$ into 7 equal intervals and then we ``transfer'' these intervals back to $\Z_+$ in the same way as in the previous subsection; we omit the details. Each of the resulting intervals $I_{t,s}^{\varkappa, k; (j)}$, $1\le j\le 7$, satisfies 
	$$
	8\widetilde{I}_{t,s}^{\varkappa, k; (j)}\subset 2 \widetilde{I}_{t,s}^{\varkappa,k}\subset \widetilde{I}_{t,s}^\varkappa
	$$
	(here the second inclusion follows from the properties of Whitney decomposition while the first one is verified by a simple computation: for example, if we divide the segment $J=[-7,7]$ into 7 segments of equal length $J^{(j)}$, $1\le j\le 7$, then $J^{(7)}=[5,7]$ and $8J^{(7)}=[-2,14]$) and it is enough to prove the estimate for the square function with respect to each of the 7 families $\{I_{t,s}^{\varkappa, k; (j)}\}_{t,s, \varkappa, k}$. Therefore (in order to make notation less bulky) we may assume without loss of generality that
	\begin{equation}
		\label{int_1}
		8\widetilde{I}_{t,s}^{\varkappa,k}\subset \widetilde{I}_{t,s}^{\varkappa}
	\end{equation}
	and that for any fixed $t$ and $\varkappa\in\Omega_t$ we have
	\begin{equation}
		\label{int_2}
		\sum_{s, k} \mathbbm{1}_{8\widetilde{I}_{t,s}^{\varkappa,k}}(x)\leq C
	\end{equation}
	for all $x$.
	
	Note that after these procedures certain intervals could become very small. If an interval $\widetilde{I}_{t,s}^{\varkappa,k}$ does not contain points of the form $j/p_t$, then we simply ignore it. For technical reasons, it would be more convenient for us also to deal separately with indices $t, \varkappa, s, k$ such that $\widetilde{I}_{t,s}^{\varkappa, k}$ contains only one point of the form $j/p_t$. One of the ways to do it is to apply the decomposition described in Section 2 to each of such intervals intervals $I_{t,s}^{\varkappa, k}$ once again. It is easy to see that for such particular intervals the inequality \eqref{main_refined} will be reduced to the inequalities of the form \eqref{main_first} and \eqref{main_third} which were already proved. Therefore, we can assume that all our intervals contain at least two points of the form $j/p_t$ and prove our estimate under such assumption.
	
	Now for any fixed $r\in\mathbb{Z}$ we renumber the intervals $\widetilde{I}_{t,s}^{\varkappa,k}$ with lengths between $2^r/p_t$ and $2^{1+r}/p_t$. Note that due to the discussion above we have $r\ge 0$. So, let $\{\widetilde{I}_{t,s}^{\varkappa,r}\}_s$ be the collection of our intervals such that $2^r/p_t\le |\widetilde{I}_{t,s}^{\varkappa,r}| < 2^{1+r}/p_t$. Also let $n_{t,s}^{\varkappa,r}$ be the least positive integer such that $\frac{2^r n_{t,s}^{\varkappa,r}}{p_t}\in \widetilde{I}_{t,s}^{\varkappa,r}$. Now we define 
	$$
	\phi_{t,s}^{\varkappa,r}(x)=\phi(\frac{p_tx}{2^r}-n_{t,s}^{\varkappa, r}).
	$$
	Clearly this function has the following properties:
	\begin{gather}
		\label{smooth_function_1}
		\supp\phi_{t,s}^{\varkappa,r}\subset 8\widetilde{I}_{t,s}^{\varkappa,r};\\
		\label{smooth_function_2}
		\phi_{t,s}^{\varkappa,r} \equiv 1 \ \text{on} \ \widetilde{I}_{t,s}^{\varkappa,r}.
	\end{gather}
	
	Then we should define the ``smooth versions'' of the operators $P_{I_{t,s}^{\varkappa,r}}$. We do it with the help of Vilenkin--Fourier coefficients: if $n=n_0m_0+\ldots + n_tm_t + \varkappa_1$ for $\varkappa_1\neq\varkappa$ then we set $\widehat{Q_{t,s}^{\varkappa,r} f}(n)=0$; for $n=n_0m_0+\ldots + n_tm_t + \varkappa$ we set
	$$
	\widehat{Q_{t,s}^{\varkappa,r} f}(n)=\phi_{t,s}^{\varkappa,r} (\frac{n_t}{p_t}) \widehat{f}(n).
	$$
	
	The following lemma from \cite{Young_2} helps us to replace the operators $P_{I_{t,s}^{\varkappa,r}}$ by $Q_{t,s}^{\varkappa,r}$ in the definition of the square function.
	
	\begin{lemma}
		For arbitrary functions $\{f_l\}$ and arbitrary positive integers $\{j_l\}$ the following inequality holds for $1<p<\infty$:
		$$
		\Big\| \Big( \sum_l |P_{[0,j_l]}f_l|^2\Big)^{1/2} \Big\|_{L^p} \lesssim \Big\| \Big( \sum_l |f_l|^2 \Big)^{1/2} \Big\|_{L^p}.
		$$
	\end{lemma}
	An obvious consequence of this lemma is that for arbitrary intervals $\{I_l\}$ in $\Z_+$ we have
	$$
	\Big\| \Big( \sum_l |P_{I_l}f_l|^2\Big)^{1/2} \Big\|_{L^p} \lesssim \Big\| \Big( \sum_l |f_l|^2 \Big)^{1/2} \Big\|_{L^p}.
	$$

	It follows from \eqref{smooth_function_2} that $P_{I_{t,s}^{\varkappa,r}}Q_{t,s}^{\varkappa,r}=P_{I_{t,s}^{\varkappa,r}}$. Therefore, the above lemma implies the estimate
	$$
	\Big\| \Big( \sum_t \sum_{\varkappa\in\Omega_t} \sum_r \sum_s |P_{I_{t,s}^{\varkappa,r}} f|^2 \Big)^{1/2} \Big\|_{L^p} \lesssim \Big\| \Big( \sum_t \sum_{\varkappa\in\Omega_t} \sum_r \sum_s |Q_{t,s}^{\varkappa,r} f|^2 \Big)^{1/2} \Big\|_{L^p}.
	$$
	Hence it is enough to prove the following estimate for $2\le p< \infty$:
	\begin{equation}
		\label{smooth_op}
		\Big\| \Big( \sum_t \sum_{\varkappa\in\Omega_t} \sum_r \sum_s |Q_{t,s}^{\varkappa,r} f|^2 \Big)^{1/2} \Big\|_{L^p}\lesssim \|f\|_{L^p}.
	\end{equation}
	Note that for $p=2$ this is still an immediate consequence of the Plancherel theorem. Indeed, from \eqref{int_1} and \eqref{smooth_function_1} it follows that for $(t_1,\varkappa_1)\neq (t_2, \varkappa_2)$ the functions $Q_{t_1,s_1}^{\varkappa_1,r_1} f$ and $Q_{t_2,s_2}^{\varkappa_2,r_2} f$ are orthogonal in $L^2$ (we use here that the initial intervals $\widetilde{I}_{t,s}^\varkappa$ which are constructed in the previous subsection are pairwise disjoint) and for fixed $t$ and $\varkappa \in\Omega_t$ we use the property \eqref{int_2}.
	
	It follows from the definition that $Q_{t,s}^{\varkappa,r} f$ can be written as follows:
	$$
	Q_{t,s}^{\varkappa,r} f = w_{\varkappa}\sum_{u\in\mathbb{Z}} \phi_{t,s}^{\varkappa,r}(u/p_t) \Delta_{t,u}[w_{\varkappa}^{-1}f].
	$$
	Here the summation over all $u\in\Z$ is convenient due to technical reasons; note that from  \eqref{int_1} and \eqref{smooth_function_1} it follows that for all $u\not\in \{1, 2, \ldots p_t-1\}$ the corresponding summands are equal to zero and therefore it is not important how we define $\Delta_{t,u}$ for such values of $u$ (although it is natural to think that $\Delta_{t,0}=\E_t$ and $\Delta_{t,u}=\Delta_{t, u\,\mathrm{mod}\, p_t}$).
	
	Let us put 
	$$
	\widetilde{Q}_{t,s}^{\varkappa,r} f = \sum_{u\in\mathbb{Z}} \phi_{t,s}^{\varkappa,r}(u/p_t) \Delta_{t,u}[w_{\varkappa}^{-1}f].
	$$
	Our goal is to prove the inequality \eqref{smooth_op} with $Q_{t,s}^{\varkappa,r}$ replaced by $\widetilde{Q}_{t,s}^{\varkappa,r}$.
	
	Now we will further modify the operators $\widetilde{Q}_{t,s}^{\varkappa,r}$ (in a similar way to \cite{RdF}). The definition implies that the function $\widetilde{Q}_{t,s}^{\varkappa,r}$ is constant on the intervals from $\widetilde{\F}_{t+1}$. Let us take $\widetilde{\omega}\in\widetilde{\F}_t$ and number all intervals from $\widetilde{\F}_{t+1}$ contained in it as in the formula \eqref{enum_int}. Then we put
	$$
	(\mathcal{R}_{t,s}^{\varkappa,r}f)|_{e_n}=e^{-2\pi i 2^r n_{t,s}^{\varkappa,r}n/p_t}\widetilde{Q}_{t,s}^{\varkappa,r}f. 
	$$
	We also introduce the operator $R$ that maps $f$ into an $\ell^2$-valued function in a following way:
	$$
	Rf=(\mathcal{R}_{t,s}^{\varkappa,r}f)_{t,\varkappa,r,s}.
	$$
	We will prove that this is a bounded linear operator from $L^p$ to $L^p(\ell^2)$ for $2<p<\infty$. Note that since $|\mathcal{R}_{t,s}^{\varkappa,r}f|= |\widetilde{Q}_{t,s}^{\varkappa,r} f| = |Q_{t,s}^{\varkappa,r} f|$ the operator $G$ is still $L^2$-bounded.
	
	\subsection{Sharp maximal function: reduction to an estimate on one cyclic group}
	
	For an $\ell^2$-valued function $g$ we define its sharp maximal function:
	$$
	g^{\#}(x)=\sup_{\mathscr{F}\ni\omega\ni x} \Big( \frac{1}{|\omega|} \int_\omega |g(x)-g_\omega|_{\ell^2}\, dx \Big),
	$$ 
	where $g_\omega=\frac{1}{|\omega|}\int_\omega g$. It is well-known (and easy to see) that
	$$
	\sup_{\mathscr{F}\ni\omega\ni x} \inf_{c\in\ell^2} \Big( \frac{1}{|\omega|} \int_\omega |g(x)-c|_{\ell^2}\, dx \Big)\asymp g^\# (x).
	$$
	
	We will show the following pointwise estimate:
	\begin{equation}
		\label{to_prove_maxim_func}
		(Rf)^\# (x)\lesssim (M_2 f)(x).
	\end{equation}
	It implies our inequality because the operator $M_2$ is bounded in $L^p$ for $p>2$ and for the sharp maximal function we have the following lemma.
	
	\begin{lemma}
		\label{sharp_from_below}
		For an $\ell^2$-valued function $g=(g_l)$ we have
		$$
		\int (M_1 g)^p \lesssim \int (g^\#)^p + \Big(\int |g|_{\ell^2}\Big)^p.
		$$
	\end{lemma}
	Here $M_1 g = M_1[|g|_{\ell^2}]$.
	
	Such estimates from below for the $L^p$-norm of sharp maximal function are well-known (in a classical setting of $\R^n$ and dyadic maximal functions): see for example \cite[Theorem 7.16]{MusSchlag} or \cite[pp. 153--154]{Stein}; they follow from a so-called good-$\lambda$ inequality. As for our case, the more general version of such estimate is proved in \cite{Young_weights} for the scalar-valued functions. The proof of our statement for $\ell^2$-valued functions is in fact easier: the proof of the standard case should simply be modified by using the appropriate version of Calderon--Zygmund decomposition (the one which was formulated in Lemma~\ref{Cal-Zyg}). 
	
	Now the $L^p$-boundedness of the operator $R$ follows from the estimate $\eqref{to_prove_maxim_func}$ in a straightforward way: we use Lemma \ref{sharp_from_below}, \eqref{to_prove_maxim_func} and the $L^2$-boundedness of the operator $R$ and write
	\begin{multline*}
		\|Rf\|_{L^p(\ell^2)}\leq \|M_1[Rf]\|_{L^p}\lesssim \|(Rf)^\#\|_{L^p} + \|Rf\|_{L^1(\ell^2)}\\ \lesssim \|M_2 f\|_{L^p} + \|Rf\|_{L^2(\ell^2)}\lesssim \|f\|_{L^p} + \|f\|_{L^2} \lesssim \|f\|_{L^p}.
	\end{multline*}
	
	Let us focus on the proof of the inequality \eqref{to_prove_maxim_func}. Fix a generalized interval $\omega\in\mathscr{F}_k$, $\omega\subset \widetilde{\omega}\in\widetilde{\F}_k$. We should prove that for some choise of constants $c_{t,s}^{\varkappa,r}$ the following inequality holds for $x\in\omega$:
	$$
	\frac{1}{|\omega|} \int_\omega \Big( \sum_t \sum_{\varkappa\in\Omega_t} \sum_r\sum_s |\mathcal{R}_{t,s}^{\varkappa,r} f - c_{t,s}^{\varkappa,r}|^2 \Big)^{1/2}\lesssim M_2 f(x).
	$$
	The terms with $t<k$ can be estimated easily: in this case $\mathcal{R}_{t,s}^{\varkappa,r} f$ is constant on $\widetilde{\omega}$ and thus we put $c_{t,s}^{\varkappa,r}=(\mathcal{R}_{t,s}^{\varkappa,r} f)|_{\widetilde{\omega}}$ in order to eliminate these terms.
	
	Now suppose that $t>k$. In this case we put $c_{t,s}^{\varkappa,r}=0$ and use the $L^2$-boundedness of our operator. Indeed, it follows from Lemma~\ref{orthogonality} and the definition that if $e\in\widetilde{\mathcal{F}}_{k+1}$ then $\mathbbm{1}_e Q_{t,s}^{\varkappa,r} f = Q_{t,s}^{\varkappa,r}[\mathbbm{1}_ef]$ and therefore
	$$
	\int_e  \sum_{t>k} \sum_{\varkappa\in\Omega_t} \sum_r\sum_s |Q_{t,s}^{\varkappa,r}f|^2 \lesssim \int_e |f|^2.
	$$
	Hence (it is important to recall here that $|Q_{t,s}^{\varkappa,r} f |=|\mathcal{R}_{t,s}^{\varkappa,r} f|$) we have:
	\begin{multline*}
		\frac{1}{|\omega|} \int_\omega \Big( \sum_{t>k} \sum_{\varkappa\in\Omega_t} \sum_r\sum_s |\mathcal{R}_{t,s}^{\varkappa,r} f|^2 \Big)^{1/2}\leq \Big( \frac{1}{|\omega|} \int_\omega \sum_{t>k} \sum_{\varkappa\in\Omega_t} \sum_r\sum_s |Q_{t,s}^{\varkappa,r}f|^2  \Big)^{1/2}\\ \lesssim \Big( \frac{1}{|\omega|} \int_\omega |f|^2 \Big)^{1/2}\leq M_2 f(x).
	\end{multline*}
	
	Now we only should estimate the terms with $t=k$. That is, our goal is to prove the following inequality:
	\begin{equation}
		\label{to_prove_fin}
		\frac{1}{|\omega|} \int_\omega \Big( \sum_{\varkappa\in\Omega_k} \sum_r \sum_s |\mathcal{R}_{k,s}^{\varkappa,r}f - c_{k,s}^{\varkappa,r}|^2 \Big)^{1/2} \lesssim M_2 f(x).
	\end{equation}
	
	This is basically an estimate for the functions on one cyclic group $\Z_{p_k}$ (in the same sense as in Subsection 3.3). Put $f_\varkappa=P_{J_\varkappa} f$ where $J_\varkappa$ is the following interval:
	$$
	J_\varkappa\sim
	\begin{pmatrix}
		m_0 & m_1 & \ldots & m_{k-1} & m_{k}               & m_{k+1} & \ldots & m_N\\
		*   &   * & \ldots & *   & * & \varkappa_{k+1} & \ldots & \varkappa_N 
	\end{pmatrix}.
	$$ 
	Then $\mathcal{R}_{k,s}^{\varkappa,r} f = \mathcal{R}_{k,s}^{\varkappa,r} f_\varkappa$. We also put $g_\varkappa=w_{\varkappa}^{-1}f_\varkappa$. We again number all segments from $\widetilde{\F}_{k+1}$ which are contained in $\widetilde{\omega}$:
	$$
	\widetilde{\omega}=e_0\cup\ldots\cup e_{p_k-1}.
	$$ 
	The functions $g_\varkappa$ are constants on each of the intervals $e_l$ and therefore we can consider them as functions on $\Z_{p_k}=\{0, 1, \ldots p_k-1\}$. We will write $g_\varkappa (l)$ instead of $g_\varkappa|_{e_l}$. Then we have
	$$
	(\Delta_{k,u} g_{\varkappa})|_{e_n}=\frac{1}{p_k} e^{2\pi i n u/p_k}\sum_{m=0}^{p_k-1} e^{-2\pi i u m/p_k} g_{\varkappa}(m).
	$$
	Therefore we can compute $(\widetilde{Q}_{k,s}^{\varkappa,r} f)|_{e_n}$ in the following way:
	\begin{multline*}
		(\widetilde{Q}_{k,s}^{\varkappa,r} f)|_{e_n}=\frac{1}{p_k}\sum_{m=0}^{p_k-1} g_{\varkappa}(m)\sum_{u\in\Z}\phi_{k,s}^{\varkappa,r}(u/p_k)e^{2\pi i u(n-m)/p_k}\\=\frac{1}{p_k}\sum_{m=0}^{p_k-1}g_\varkappa(m)\sum_{u\in\Z}\phi\big( \frac{u}{2^r}-n_{k,s}^{\varkappa,r}\big) e^{2\pi i u(n-m)/p_k}=\ldots
	\end{multline*}
	We put $l=u-2^r n_{k,s}^{\varkappa,r}$ and continue the computation:
	$$
	\ldots= \frac{1}{p_k}\sum_{m=0}^{p_k-1} g_{\varkappa}(m)\sum_{l\in\Z}\phi\big(\frac{l}{2^r}\big) e^{2\pi i(l+2^r n_{k,s}^{\varkappa,r})(n-m)/p_k}.
	$$
	Now recall that $(\mathcal{R}_{k,s}^{\varkappa,r}f)|_{e_n}=e^{-2\pi i 2^r n_{k,s}^{\varkappa,r}n/p_k}\widetilde{Q}_{k,s}^{\varkappa,r}f$ and hence
	$$
	(\mathcal{R}_{k,s}^{\varkappa,r}f)|_{e_n}=\frac{1}{p_k} \sum_{m=0}^{p_k-1} g_\varkappa(m) e^{-2\pi i 2^r n_{k,s}^{\varkappa,r} m /p_k}\sum_{l\in\Z} \phi\big( \frac{l}{2^r} \big)e^{2\pi i l(n-m)/p_k}.
	$$
	Consider now the following kernel on $\Z_{p_k}\times\Z_{p_k}$:
	$$
	K_{k,s}^{\varkappa,r}(n,m)=\frac{1}{p_k}e^{-2\pi i 2^r n_{k,s}^{\varkappa,r}m/p_k}\sum_{l\in\Z}\phi\big( \frac{l}{2^r}\big)e^{2\pi i l(n-m)/p_k}.
	$$
	Also let $T_{k,s}^{\varkappa,r}$ be the operator which maps a function $g$ on $\Z_{p_k}$ to the following function on $\Z_{p_k}$:
	$$
	T_{k,s}^{\varkappa,r}[g](n)=\sum_{m\in\Z_{p_k}} K_{k,s}^{\varkappa,r}(n,m)g(m).
	$$
	Suppose also that
	$$
	\omega=\bigcup_{q\in I} e_q,
	$$
	where $e_q\in \widetilde{\F}_{k+1}$ are the atoms of $\F_{k+1}$ and $I$ is an arc in $\Z_{p_k}$ (we can consider the elements of $\Z_{p_k}$ as points on the unit circle $\{e^{2\pi i l /p_k}\}$, $0\le l\le p_k-1$). Then the left-hand side of the inequality \eqref{to_prove_fin} can be rewritten in a following way:
	$$
	\frac{1}{|I|}\sum_{z\in I} \Big( \sum_{\varkappa,r,s} |T_{k,s}^{\varkappa,r}[g_\varkappa](z)-c_{k,s}^{\varkappa,r}|^2 \Big)^{1/2}.
	$$
	Here $|I|$ denotes the number of elements in $I$. 
	
	Besides that, the functions $f_{\varkappa_0}$ and $f_{\varkappa_1}$ are orthogonal in $L^2(e_l)$ for each $l$ if $\varkappa_0\neq\varkappa_1$. Therefore we see that 
	$$
	\sum_{\varkappa\in\Omega_k} |g_\varkappa(l)|^2\leq \frac{1}{|e_l|}\int_{e_l} |f|^2
	$$ 
	and hence if $x\in e_n$ then
	$$
	\mathcal{M}_2[(g_\varkappa)_\varkappa](n)\leq M_2f(x).
	$$
	Here on the left $\mathcal{M}_2$ denotes the Hardy--Littlewood maximal function on $\Z_{p_k}$; its definition is self-explanatory.
	
	Now it remains to prove the following inequality for arbitrary arc $I\subset\Z_{p_k}$ and $n\in I$:
	\begin{equation}
		\label{to_prove_cycl}
		\frac{1}{|I|}\sum_{z\in I} \Big( \sum_{\varkappa,r,s} |T_{k,s}^{\varkappa,r}[g_\varkappa](z)-c_{k,s}^{\varkappa,r}|^2 \Big)^{1/2} \lesssim \mathcal{M}_2[(g_\varkappa)_\varkappa](n). 
	\end{equation}
	Here $g_\varkappa$ are arbitrary functions on $\Z_{p_k}$ and we will specify the choice of constants $c_{k,s}^{\varkappa,r}$ later.
	
	Is is important to note that the operator $T$ that acts on $\ell^2$-valued functions on $\Z_{p_k}$ as follows:
	$$
	T[(g_\varkappa)_\varkappa]=(T_{k,s}^{\varkappa,r}g_\varkappa)_{s}^{\varkappa,r}
	$$
	(that is, $T$ maps an $\ell^2$-valued function $(g_\varkappa)$ to an $\ell^2$-valued function indexed by $\varkappa$, $r$ and $s$) is $L^2$-bounded.
	
	\subsection{The proof of an estimate on one cyclic group}
	
	Let us prove the inequality \eqref{to_prove_cycl}. It is worth noting that this inequality is an analogue of the estimate for sharp maximal function that was proved in \cite{RdF} (for functions on $\Z_{p_k}$ instead of $\R$). It might be possible that this estimate can be derived from the one in \cite{RdF} by using some sort of transference principle. However, it seems that such reduction of one inequality to another is at least not immediate; besides that, in \cite{RdF} the inequality is proved only for one scalar-valued function instead of the $\ell^2$-valued function $(g_\varkappa)$ (however, it is not a serious obstacle). Thus we choose to adapt the proof from Rubio de Francia's original paper to our context of cyclic groups.
	
	First of all, since now we work with just one group $\Z_{p_k}$, we eliminate the index $k$ from our notation. Also we put $\widetilde{g}_\varkappa = g_\varkappa\mathbbm{1}_{3I}$ where $3I$ is an arc in $\Z_p$ with the same center as $I$ and such that $|3I|=3|I|$ if $|I|\le p/3$ (and if $|I|>p/3$ then $3I=\Z_p$).
	
	We fix arbitrary $n\in I$, put 
	$$
	c_{s}^{\varkappa,r}=\sum_{m\not\in 3I}K_{s}^{\varkappa,r}(n,m)g_\varkappa(m)
	$$
	and write:
	$$
	T_{s}^{\varkappa,r}[g_\varkappa](z)-c_{s}^{\varkappa,r}=T_s^{\varkappa,r}[\widetilde{g}_\varkappa](z)+\sum_{m\not\in 3I}g_\varkappa(m)(K_s^{\varkappa,r}(z,m)-K_s^{\varkappa,r}(n,m)).
	$$
	Let us take arbitrary functions $\xi_s^{\varkappa,r}$ on $\Z_p$ such that $\sum_{\varkappa,r,s}|\xi_s^{\varkappa,r}(z)|^2 \le 1$ for all $z$. Then the left-hand side of the inequality \eqref{to_prove_cycl} can be estimated as follows:
	\begin{multline*}
		\frac{1}{|I|}\sum_{z\in I} \Big( \sum_{\varkappa,r,s} |T_{s}^{\varkappa,r}[g_\varkappa](z)-c_{s}^{\varkappa,r}|^2 \Big)^{1/2}  \leq \frac{1}{|I|} \sum_{z\in I}\Big( \sum_{\varkappa,r,s} |T_s^{\varkappa,r}[\widetilde{g}_\varkappa](z)|^2 \Big)^{1/2} \\+\sup_\xi \frac{1}{|I|}\Big| \sum_{\substack{z\in I\\z\neq n}} \sum_{\varkappa, r, s} \xi_s^{\varkappa,r}(z)\Big( \sum_{m\not\in 3I}g_\varkappa(m)(K_s^{\varkappa,r}(z,m)-K_s^{\varkappa,t}(n,m)) \Big) \Big|=: (I)+(II).
	\end{multline*}
	
	The estimate of quantity $(I)$ is easy: we simply use Cauchy--Schwarz inequality and $L^2$-boundedness of operator $T$:
	\begin{multline*}
		\frac{1}{|I|} \sum_{z\in I}\Big( \sum_{\varkappa,r,s} |T_s^{\varkappa,r}[\widetilde{g}_\varkappa](z)|^2 \Big)^{1/2} \leq \Big( \frac{1}{|I|}\sum_{z\in I}\sum_{\varkappa,r,s} |T_s^{\varkappa,r}[\widetilde{g}_\varkappa](z)|^2 \Big)^{1/2}\\ \leq   \Big( \frac{1}{|I|}\sum_{z\in \Z_p}\sum_{\varkappa,r,s} |T_s^{\varkappa,r}[\widetilde{g}_\varkappa](z)|^2 \Big)^{1/2}\lesssim \Big( \frac{1}{|I|} \sum_{z\in\Z_p}\sum_\varkappa |\widetilde{g}_\varkappa(z)|^2 \Big)^{1/2}\\=\Big( \frac{1}{|I|} \sum_{z\in 3I}\sum_\varkappa |g_\varkappa(z)|^2 \Big)^{1/2}\leq \sqrt{3}\mathcal{M}_2[(g_\varkappa)_\varkappa](n).
	\end{multline*}
	
	Let us now prove the bound for the quantity $(II)$ for each $\xi$. We write:
	\begin{multline*}
		\frac{1}{|I|}\Big| \sum_{\substack{z\in I\\z\neq n}} \sum_{\varkappa, r, s} \xi_s^{\varkappa,r}(z)\Big( \sum_{m\not\in 3I}g_\varkappa(m)(K_s^{\varkappa,r}(z,m)-K_s^{\varkappa,r}(n,m)) \Big) \Big| \\ \leq \frac{1}{|I|} \sum_{\substack{z\in I\\z\neq n}} \sum_{m\not\in 3I} \sum_{\varkappa} \Big( |g_\varkappa(m)| \Big| \sum_{r,s} \xi_s^{\varkappa,r}(z)(K_s^{\varkappa,r}(z,m)-K_s^{\varkappa,r}(n,m)) \Big| \Big)\leq \ldots.
	\end{multline*}
	For $x, z\in \Z_p$ we denote
	$$
	I_k(x,z)=\{y\in \Z_p:\ 2^k\dist_p(x,z) < \dist_p(y,z)\le 2^{k+1}\dist_p(x,z) \}.
	$$
	Recall that $\dist_p$ is a distance on $\Z_p$ that is defined as follows: for $x,y\in \Z_p$, we consider an integer $c$ such that $c\equiv x-y$ modulo $p$ and $|c|\le p/2$; then $\dist_p(x,y)=|c|$.
	
	Using this notation, we continue the estimate:
	\begin{multline*}
		\ldots\leq\frac{1}{|I|} \sum_{\substack{z\in I\\z\neq n}} \sum_{k\ge 1} \sum_{m\in I_k(n,z)} \sum_{\varkappa} \Big( |g_\varkappa(m)| \Big| \sum_{r,s} \xi_s^{\varkappa,r}(z)(K_s^{\varkappa,r}(z,m)-K_s^{\varkappa,r}(n,m)) \Big| \Big) \\ \leq \frac{1}{|I|} \sum_{\substack{z\in I\\z\neq n}} \sum_{k\ge 1} \sum_\varkappa \bigg[ \Big( \sum_{m\in I_k(n,z)} |g_\varkappa(m)|^2 \Big)^{1/2} \\ \times \Big( \sum_{m\in I_k(n,z)} \Big| \sum_{r,s} \xi_s^{\varkappa,r}(z)(K_s^{\varkappa,r}(z,m) - K_s^{\varkappa,r}(n,m)) \Big|^2 \Big)^{1/2} \bigg]\leq\ldots
	\end{multline*}
	
	Now we need the following lemma.
	
	\begin{lemma}\label{main_est_kernel}
		There exist real numbers $\alpha>1$ and $A>0$ such that for any $\varkappa$ and any numbers $\lambda_s^r$ we have
		\begin{equation}
			\label{key_est_kernel}
			\sum_{y\in I_k(x,z)} \Big| \sum_{r,s} \lambda_s^r (K_s^{\varkappa,r}(z,y)-K_s^{\varkappa,r}(x,y)) \Big|^2 \leq \frac{A^2 2^{-\alpha k} \sum_{r,s}|\lambda_s^r|^2}{\dist_p(x,z)}.
		\end{equation}
	\end{lemma}
	
	We will prove this lemma in the next subsection. Now let assume that it holds and continue our estimate.
	\begin{multline*}
		\cdots\leq \frac{1}{|I|} \sum_{\substack{z\in I\\z\neq n}} \sum_{k\ge 1} \sum_\varkappa \bigg[ \Big( \sum_{m\in I_{k}(n,z)} |g_\varkappa(m)|^2 \Big)^{1/2} \Big( A^2 \frac{2^{-\alpha k} \sum_{r,s}|\xi_s^{\varkappa, r}(z)|^2}{\dist_p(n,z)} \Big)^{1/2}  \bigg] \\ \leq \frac{A}{|I|} \sum_{\substack{z\in I\\z\neq n}}\sum_{k\ge 1} \bigg[ \Big( \sum_{m\in I_k(n,z)} \sum_\varkappa |g_\varkappa(m)|^2 \Big)^{1/2} \Big( \frac{2^{-\alpha k} \sum_{\varkappa,r,s}|\xi_s^{\varkappa,r}(z)|^2}{\dist_p(n,z)} \Big)^{1/2} \bigg] \\ \leq\frac{A}{|I|} \sum_{\substack{z\in I\\z\neq n}} \sum_{k\ge 1} \frac{2^{-\alpha k/2}}{\dist_p(n,z)^{1/2}} \Big( \sum_{m\in I_k(n,z)} \sum_\varkappa |g_\varkappa(m)|^2 \Big)^{1/2}\lesssim \ldots
	\end{multline*}
	In order to pass to the last line we used that $\sum_{\varkappa,r,s}|\xi_s^{\varkappa,r}(z)|^2 \le 1$ for all $z$. 
	
	Note that $I_k(n,z)\subset\{m:\, \dist_p(m,z)\leq 2^{k+1}\dist_p(n,z)\}$ and this last set is an arc in $\Z_p$ that contains $n$ and consists of $\asymp 2^k\dist_p(n,z)$ elements. Therefore we have
	$$
	\Big( \sum_{m\in I_k(n,z)} \sum_\varkappa |g_\varkappa(m)|^2 \Big)^{1/2}\lesssim 2^{k/2}\dist_p(n,z)^{1/2}\mathcal{M}_2[(g_\varkappa)_\varkappa](n).
	$$
	Using this inequality, we finish our estimate:
	$$
	\ldots\lesssim \frac{A}{|I|}\sum_{z\in I} \sum_{k\ge 1} 2^{(1-\alpha)k/2}\mathcal{M}_2[(g_\varkappa)_\varkappa](n)\lesssim \mathcal{M}_2[(g_\varkappa)_\varkappa](n).
	$$
	
	\subsection{The proof of lemma \ref{main_est_kernel} (the basic estimate on the kernel)}
	
	It remains only to prove lemma \ref{main_est_kernel}. Since we need to prove the inequality for each $\varkappa$, we omit index $\varkappa$ in all our notation. Without loss of generality we can assume that $\sum_{r,s}|\lambda_r^s|^2=1$.
	
	Recall the definition of $K_s^r$:
	$$
	K_s^r(x,y)=\frac{1}{p}e^{-2\pi i 2^r n_s^r m/p}\sum_{l\in\Z}\phi\big( \frac{l}{2^r} \big)e^{2\pi i l (x-y)/p},
	$$
	where $x, y\in\Z_p$.
	Denote 
	$$
	\psi_r(t)=\sum_{l\in\Z}\phi\big(\frac{l}{2^r}\big)e^{2\pi i l t/p}.
	$$
	In this formula $t$ can be an arbitrary real number. However, since obviously we have $\psi_r(t)=\psi_r(t+p)$, we can substitute the elements of $\Z_p$ in the place of $t$ in this formula. 
	
	We estimate the square root of the left-hand side in the formula \eqref{key_est_kernel} as follows:
	\begin{multline*}
		\Big(\sum_{y\in I_k(x,z)} \Big| \sum_{r,s} \lambda_s^r (K_s^{r}(z,y)-K_s^{r}(x,y)) \Big|^2\Big)^{1/2}\\=\frac{1}{p}\Big( \sum_{y\in I_k(x,z)} \Big| \sum_r \Big[ (\psi_r(x-y)-\psi_r(z-y)) \Big( \sum_s \lambda_s^r e^{-2\pi i 2^r n_s^r y/p} \Big)\Big] \Big|^2 \Big)^{1/2}\\ \leq \frac{1}{p} \sum_r \Big( \sum_{y\in I_k(x,z)} \Big[ |\psi_r(x-y)-\psi_r(z-y)|^2 \Big| \sum_s \lambda_s^r e^{-2\pi i 2^r n_s^r y/p} \Big|^2 \Big] \Big)^{1/2}\\ \leq \frac{1}{p} \sum_r \Big[ \sup_{t\in I_k(x,z)} |\psi_r(x-t)-\psi_r(z-t)|\Big( \sum_{y\in I_k(x,z)} \Big| \sum_s \lambda_s^r e^{-2\pi i 2^r n_s^r y/p} \Big|^2 \Big)^{1/2}\Big]\lesssim \ldots
	\end{multline*}
	
	Let us now estimate the sum over $y\in I_k(x,z)$ in this expression. Recall that for a fixed $r$ the numbers $\{n_s^r\}$ are pairwise different and $0<n_s^r<p/2^r$ (it follows directly from the definition of these numbers in subsection 4.4). It is convenient for us to divide the family of kernels $\{K_s^r\}$ into two families: with $0<n_s^r \le p/2^{1+r}$ and with $p/2^{1+r}<n_s^r<p/2^r$. Clearly, we can do it since it is enough to prove lemma \ref{main_est_kernel} for each of these families. Then we apply the following elementary lemma.
	
	\begin{lemma}
		Suppose that $a, b\in\Z$ and $[a,b]\subset\Z$ contains $[\frac{p}{2^r}]$ integer points \emph($[x]$ denotes the integer part of $x$\emph). Then 
		$$
		\sum_{y=a}^b \Big| \sum_{0<j\le\frac{p}{2^{r+1}}} \lambda_j e^{-2\pi i 2^r j y/p} \Big|^2 \leq C \frac{p}{2^{1+r}} \sum_{0<j\le\frac{p}{2^{r+1}}} |\lambda_j|^2.
		$$
	\end{lemma}
	
	\begin{proof}
		We have:
		\begin{multline*}
			\sum_{y=a}^b \Big| \sum_{0<j\le\frac{p}{2^{r+1}}} \lambda_j e^{-2\pi i 2^r j y/p} \Big|^2 = \sum_{y=a}^b \Big( \sum_j |\lambda_j|^2 + \sum_{m}\sum_{j\neq m} \lambda_j\bar{\lambda}_k e^{2\pi i 2^r(m-j)y/p} \Big) \\ \lesssim \Big(\frac{p}{2^{1+r}} \sum_j |\lambda_j|^2\Big) + \sum_m \sum_{j\neq m} (|\lambda_m|^2 + |\lambda_j|^2) \Big| \sum_{y=a}^b e^{2\pi i 2^r(m-j)y/p} \Big|.
		\end{multline*}
		
		The lemma will be proved once we show that the following inequality holds:
		$$
		\Big| \sum_{y=a}^b e^{2\pi i 2^r(m-j)y/p} \Big|\leq 1.
		$$
		Let $p=2^r g + u$, $0\le u <2^r$. Then 
		\begin{multline*}
			\Big| \sum_{y=a}^b e^{2\pi i 2^r(m-j)y/p} \Big|= 	\Big| \sum_{y=0}^{g-1} e^{2\pi i 2^r(m-j)y/p} \Big|\\ =\Big| \frac{e^{2\pi i (m-j)\frac{2^r g}{2^r g + u}}-1}{e^{2\pi i (m-j)\frac{2^r}{2^rg+u}}-1} \Big|=\Big| \frac{e^{2\pi i (m-j)\frac{u}{p}}-1}{e^{2\pi i (m-j)\frac{2^r}{p}}-1} \Big|.
		\end{multline*}
		This quantity is indeed not greater than 1 since $|m-j|<\frac{p}{2^{1+r}}$.
		
	\end{proof}
	
	Of course, the same inequality as in this lemma holds true if we change summation over $\{j:\, 0<j\le\frac{p}{2^{r+1}} \}$ to summation over the set $\{j:\, \frac{p}{2^{1+r}} < j < \frac{p}{2^r}\}$. Thus we can apply this lemma to estimate our sum over $y\in I_k(x,z)$ (by dividing $I_k(x,z)$ into several intervals with lengths $[\frac{p}{2^r}]$ and one interval with smaller length):
	$$
	\sum_{y\in I_k(x,z)} \Big| \sum_s \lambda_s^r e^{-2\pi i 2^r n_s^r y/p} \Big|^2 \lesssim 2^k\dist_p(x,z)+\frac{p}{2^r}.
	$$
	Note that we used here our assumption that $\sum |\lambda_s^r|^2=1$. Using this bound, we continue our estimate:
	\begin{equation}
		\label{FIN}
		\ldots \leq \frac{1}{p} \sum_r (2^k \dist_p(x,z)+\frac{p}{2^r})^{1/2}\sup_{t\in I_k(x,z)} |\psi_r(x-t)-\psi_r(z-t)|.
	\end{equation}
	Our next goal is to prove two different estimates for the supremum in the above formula. The first one is easy: we simply compute the derivative of the function $\psi_r$ and apply the mean value theorem. We have:
	$$
	\psi_r'(t)=\sum_{l\in\Z}\frac{2\pi i l}{p}\phi\big(\frac{l}{2^r}\big)e^{2\pi i l t/p}.
	$$
	All terms with $l\not\in [-3\cdot 2^r, 3\cdot 2^r]$ in this sum are equal to zero. Therefore, we have the following inequality:
	$$
	|\psi'_r(t)|\leq \frac{1}{p} \sum_{l=-3\cdot 2^r}^{3\cdot 2^r} |2\pi i l|\lesssim \frac{2^{2r}}{p}.
	$$
	Hence we see that 
	\begin{equation}
		\label{first_smooth_func}
		\sup_{t\in I_k(x,z)} |\psi_r(x-t)-\psi_r(z-t)|\lesssim \frac{2^{2r}}{p}\dist_p(x,z).
	\end{equation}
	
	Now let us estimate the quantity $|\psi_r(t)|$. Put $\pphi_r(u)=\phi(u/2^r)$. Then we apply the Poisson summation formula: 
	$$
	\psi_r(t)=\sum_{l\in\Z}\pphi_r(l)e^{2\pi i l t/p}=\sum_{m\in\Z}\pphi_r^\vee\big( \frac{t}{p} +m \big).
	$$
	Here $\pphi_r^\vee$ is a usual inverse Fourier transform of a function on $\R$:
	$$
	|\pphi_r^\vee (\xi)|=|2^r \phi^\vee(2^r\xi)|\lesssim \frac{2^r}{1+2^{2r}\xi^2}.
	$$
	Thus we have:
	$$
	|\psi_r(t)|\lesssim \sum_{m\in\Z}\frac{2^r}{1+2^{2r}(\frac{t}{p}+m)^2}\leq 2^{-r}\sum_{m\in\Z} \frac{1}{(m+\frac{t}{p})^2}\lesssim 2^{-r}\frac{p^2}{\dist_p(t,0)^2}.
	$$
	Since for each $t\in I_k(x,z)$ we have $\dist_p(x,t)\asymp\dist_p(z,t)\asymp 2^k \dist_p(x,z)$, the following inequality follows from here:
	\begin{equation}
		\label{second_smooth_func}
		\sup_{t\in I_k(x,z)} |\psi_r(x-t)-\psi_r(z-t)|\lesssim 2^{-r-2k}\frac{p^2}{\dist_p(x,z)^2}.
	\end{equation}
	
	In order to finish the estimate of the quantity \eqref{FIN}, we simply compare two inequalities \eqref{first_smooth_func} and \eqref{second_smooth_func}. For the summands with $r$ such that $$2^r\le 2^{-2k/3}\frac{p}{\dist_p(x,z)}=:H$$ we use the inequality \eqref{first_smooth_func}:
	\begin{multline*}
		\frac{1}{p} \sum_{r: 2^r\le H} (2^k \dist_p(x,z)+\frac{p}{2^r})^{1/2}\sup_{t\in I_k(x,z)} |\psi_r(x-t)-\psi_r(z-t)|\\
		\lesssim \frac{1}{p}\sum_{r: 2^r\le H} (2^{k/2}\dist_p(x,z)^{1/2}+2^{-r/2}p^{1/2})\frac{2^{2r}}{p}\dist_p(x,z)\\ \lesssim \frac{1}{p^2}2^{k/2} \dist_p(x,z)^{3/2}H^2 + \frac{1}{p^{3/2}}\dist_p(x,z)H^{3/2}=\frac{2^{-5k/6}+2^{-k}}{\dist_p(x,z)^{1/2}}.
	\end{multline*}
	To estimate the summands with $2^r>H$ we use the inequality \eqref{second_smooth_func}:
	\begin{multline*}
		\frac{1}{p} \sum_{r: 2^r> H} (2^k \dist_p(x,z)+\frac{p}{2^r})^{1/2}\sup_{t\in I_k(x,z)} |\psi_r(x-t)-\psi_r(z-t)|\\ \lesssim \frac{1}{p}\sum_{r: 2^r > H} (2^{k/2}\dist_p(x,z)^{1/2}+2^{-r/2}p^{1/2})2^{-r-2k}\frac{p^2}{\dist_p(x,z)^2}\\ \lesssim p\, \dist_p(x,z)^{-3/2} 2^{-3k/2} H^{-1} + p^{3/2} \dist_p(x,z)^{-2}2^{-2k}H^{-3/2}=\frac{2^{-5k/6}+2^{-k}}{\dist_p(x,z)^{1/2}}.
	\end{multline*}
	
	It remains to collect the estimates and conclude that the desired inequality \eqref{key_est_kernel} is proved with $\alpha=5/3$.
	
	\section{Weighted Littlewood--Paley inequality on cyclic groups}
	
	In this section we will prove the weighted version of Littlewood--Paley inequality for cyclic groups which we used in the proof of Lemma \ref{refinement}. We start with an exact formulation. We consider the group $\Z_p$ with the standard counting measure on it. For a function $f$ on $\Z_p$ and $E\subset\Z_p$ we denote by $P_E$ the following operator: $P_E f=(\mathbbm{1}_E\widehat{f})^\vee$.
	
	\begin{theorem}
		Suppose that $p\in\Z$, $p\ge 2$ and $h$ is an arbitrary function on $\Z_p$. Let $v\in A_2(\Z_p)$ be a Muckenhoupt weight on $\Z_p$. Suppose that $\{\widetilde{I}_s\}$ is a collection of lacunary subintervals of $(0,1)$ and $I_s$ are corresponding ``arcs'' in $\Z_p$: $j\in I_s$ if and only if $\frac{j}{p}\in\widetilde{I}_s$. Then 
		\begin{equation}
			\label{weighted_zp}
			\sum_{l\in\Z_p} \sum_s |(P_{I_s}h)(l)|^2 v(l)\leq C([v]_{A_2})\sum_{l\in\Z_p} |h(l)|^2 v(l).
		\end{equation}
	\end{theorem}
	
	It is important here that the constant $C([v]_{A_2})$ does not depend on $p$ but depends only on the quantity $[v]_{A_2}$. The definition of Muckenhoupt weights on the group $\Z_p$ is natural: $v\in A_2$ if
	$$
	\sup_{I\subset\Z_p} \Big(\frac{\sum_{l\in I} v(l)}{|I|}\Big) \Big(\frac{\sum_{l\in I} v(l)^{-1}}{|I|} \Big)=[v]_{A_2}<\infty,
	$$
	where the supremum is taken over all ``arcs'' in $\Z_p$ (as always, it is helpful to consider the elements of $\Z_p$ as points $e^{2\pi i j/p}$ on the unit circle). 
	
	It is easy to check that  this theorem implies the inequality \eqref{lp_cyclic_weights} (if we take into account the discussion after this inequality).
	
	Of course, since the weighted version of Littlewood--Paley inequality on $\R$ is known (it is proved in \cite{Kurtz}), the proof of this theorem is basically an exercise in application of the transference principle (a standard reference for some general results concerning transference is \cite{CoifWeiss}). However, we are in a slightly non-standard setting of cyclic groups and Muckenhoupt weights on them and so it is difficult to find an exact reference for the transference result we need. This is why we choose to briefly explain how to prove this theorem.
	
	\begin{proof}[Proof of Theorem 2]
		We will explain how to derive the inequality \eqref{weighted_zp} from the corresponding inequality for functions on $\Z$. First of all, we note that it is enough to prove the following estimate for arbitrary choise of signs $\eps_s=\pm 1$:
		$$
		\sum_{l\in\Z_p} \Big| \sum_s \eps_s (P_{I_s}h)(l) \Big|^2 v(l)\lesssim \sum_{l\in\Z_p} |h(l)|^2 v(l).
		$$
		The initial inequality would follow from this one if we average over all choises of $\eps_s=\pm 1$. 
		
		The second step is to linearize our estimate: due to the standard duality arguement, it is enough to prove that 
		$$
		\Big| \sum_{l\in\Z_p} g(l)\Big( \sum_s \eps_s(P_{I_s} h)(l)\Big) v(l) \Big| \lesssim \Big( \sum_{l\in\Z_p} |h(l)|^2 v(l) \Big)^{1/2} \Big( \sum_{l\in\Z_p} |g(l)|^2 v(l)\Big)^{1/2},
		$$ 
		where $g$ is an arbitrary function on $\Z_p$.
		
		Let us define the counterparts of our functions on $\Z$. For $n=mp+l\in\Z$, where $0\le l\le p-1$ we set 
		$$
		w(n)=v(l), \ H_\delta(n)=e^{-\delta|m|}h(l), \ G_\delta(n)=e^{-\delta|m|}g(l).
		$$
		Here we identified the elements of $\Z_p$ with the elements of the set $\{0, 1, \ldots, p-1\}$. We will also do so further. It is easy to see that $w\in A_2(\Z)$ and the quantity $[w]_{A_2(\Z)}$ is controlled by $[v]_{A_2(\Z_p)}$.
		
		Let us use the following inequality:
		\begin{multline}
			\label{lp_on_z}
			\Big|\sum_{n\in\Z}G_\delta(n) \Big( \eps_s \sum_s (P_{\widetilde{I}_s}H_\delta)(n) \Big)w(n)\Big| \\ \lesssim \Big( \sum_{n\in\Z} |H_\delta (n)|^2 w(n) \Big)^{1/2}  \Big( \sum_{n\in\Z} |G_\delta(n)|^2 w(n) \Big)^{1/2}.
		\end{multline}
		Here $P_{\widetilde{I}_s}$ is a Fourier multiplier with the symbol $\mathbbm{1}_{\widetilde{I}_s}$ which acts on the functions on $\Z$.
		
		This is a weighted version of Littlewood--Paley inequality for functions on $\Z$. It follows from the same inequality for the functions on $\R$ with the help of the transference method presented in \cite{BerkGill}.
		
		It is easy to compute the right-hand side in the inequality \eqref{lp_on_z}. We have:
		$$
		\sum_{n\in\Z} |H_\delta (n)|^2 w(n) = \sum_{l=0}^{p-1} \sum_{m\in\Z} e^{-2\delta |m|} |h(l)|^2 v(l)=\frac{1+e^{-2\delta}}{1-e^{-2\delta}} \sum_{l=0}^{p-1} |h(l)|^2 v(l).
		$$
		Therefore, the right-hand side of the inequality \eqref{lp_on_z} equals
		$$
		\frac{1+e^{-2\delta}}{1-e^{-2\delta}} \|h\|_{L^2(\Z_p; v)} \|g\|_{L^2(\Z_p; v)}.
		$$
		Now let us compute the left-hand side. We introduce the following notation:
		$$
		Q_\delta(\theta) = \sum_{n\in\Z} e^{-\delta |n|} e^{2\pi i n \theta}=\frac{1-r^2}{1-2e^{-\delta}\cos 2\pi\theta + e^{-2\delta}}.
		$$
		This is a Poisson kernel.
		
		We need to compute the Fourier transform of $H_\delta$ (which is a function on $\T$). It is done in a similar way as above:
		$$
		\widehat{H}_\delta (t)=\sum_{a\in\Z} e^{-2\pi i a t}H_\delta (a)=\sum_{k=0}^{p-1}\sum_{b\in\Z} e^{-2\pi i (bp+k)t} e^{-\delta |b|}h(k)=Q_\delta (pt) \sum_{k=0}^{p-1} e^{-2\pi i k t} h(k).
		$$
		Hence we have the following formula:
		$$
		(P_{\widetilde{I}_s}H_\delta)(n)=\int_{\widetilde{I}_s} Q_\delta(pt)e^{2\pi i n t}\sum_{k=0}^{p-1} e^{-2\pi i k t} h(k)\, dt.
		$$
		Finally, we arrive at the following equation for the left-hand side of \eqref{lp_on_z}:
		\begin{multline*}
			\Big|\sum_{n\in\Z}G_\delta(n) \Big(  \sum_s \eps_s (P_{\widetilde{I}_s}H_\delta)(n) \Big)w(n)\Big|\\=\Big| \sum_s \eps_s \sum_{l=0}^{p-1}  \sum_{m\in\Z} e^{-\delta |m|} g(l) v(l) \int_{\widetilde{I}_s} Q_\delta (pt) e^{2\pi i (mp+l) t} \sum_{k=0}^{p-1} e^{-2\pi i k t} h(k)\, dt  \Big|\\ = \Big| \sum_s \eps_s \sum_{l=0}^{p-1} \sum_{k=0}^{p-1} h(k) g(l) v(l) \int_{\widetilde{I}_s} Q_\delta(pt) \sum_{m\in\Z}e^{-\delta|m|} e^{2\pi i m p t} e^{2\pi i (l-k)t}\, dt \Big| \\ = \Big| \sum_s \eps_s \sum_{l=0}^{p-1} \sum_{k=0}^{p-1} h(k) g(l) v(l) \int_{\widetilde{I}_s} Q_\delta(pt)^2 e^{2\pi i (l-k) t} dt \Big| \\ = \frac{1}{p} \Big| \sum_s\eps_s \sum_{k=0}^{p-1} h(k) g(l) v(l) \int_{p\widetilde{I}_s} Q_\delta(x)^2 e^{2\pi i (l-k)x/p} dx \Big|.
		\end{multline*}
		
		It remains to note that 
		$$
		\int_0^1 |Q_\delta (x)|^2=\frac{1+e^{-2\delta}}{1-e^{-2\delta}}
		$$
		(it is a consequence of the Plancherel identity and the definition of $Q_\delta$) and hence the function
		$$
		\frac{1-e^{-2\delta}}{1+e^{-2\delta}}Q_\delta(x)^2
		$$
		is an approximate identity. Therefore, if we divide both sides in the formula \eqref{lp_on_z} by $	\frac{1+e^{-2\delta}}{1-e^{-2\delta}}$ and tend $\delta$ to zero, we get the following inequality:
		$$
		\frac{1}{p} \Big| \sum_s \eps_s \sum_{l\in\Z_p} \sum_{k\in\Z_p} \sum_{j\in I_s} h(k) g(l) e^{2\pi i (l-k)j/p}\Big| \lesssim \|h\|_{L^2(\Z_p; v)} \|g\|_{L^2(\Z_p; v)}.
		$$
		This is the same as the inequality \eqref{weighted_zp} and so the theorem is proved.
		
	\end{proof}
	
	\textbf{Acknowledgments.} This work was supported by the Ministry of Science and Higher Education of the Russian Federation, agreement no. 075-15-2022-289.
	
	The author is kindly grateful to S. V. Kislyakov for posing this problem. The author is also grateful to N. Osipov and V. Borovitskii for valuable discussions and to anonymous referees for comments which improved the presentation.


\begin{thebibliography}{99}
		\bibitem{BerkGill} {E}. {B}erkson, {T}. {A}. {G}illespie, \emph{Multipliers for weighted $L^p$-spaces, transference, and the $q$-variation of functions}, Bull. Sci. Math. \textbf{122}, 427--454 (1998)
		
		\bibitem{Bour} {J}. {B}ourgain, \emph{On square functions on the trigonometric system}, Bull. Soc. Math. Belg., \textbf{37}:1, 20--26 (1985)
		
		\bibitem{CoifWeiss} {R}. {R}. {C}oifman, {G}. {W}eiss, \emph{Transference methods in analysis}, CBMS American Mathematical Society 31. Providence, RI: American Mathematical Society (1977)
		
		\bibitem{CrMaPe} {D}. {V}. {C}ruz-{U}ribe, {J}. {M}. {M}artell and C. P\'{e}rez, \emph{Weights, Extrapolation and the Theory of Rubio de Francia}, Oper. Theory Adv. Appl. \textbf{215}, Birkhäuser/Springer Basel AG, Basel (2011)
		
		\bibitem{GarRubio} J. Garcia-Cuerva, J. L. Rubio De Francia, \emph{Weighted norm inequalities and related topics}, North-Holland
		Math. Stud., vol. 116. Notas. Math., vol. 104, North-Holland, Amsterdam (1985)
		
		\bibitem{Cl_Graf} L. Grafakos, \emph{Classical Fourier Analysis}. Third edition. Springer (2014)
		
		\bibitem{KisLP} S. V. Kislyakov, D. V. Parilov, \emph{On the Littlewood--Paley theorem for arbitrary intervals}, Zap. Nauchn. Sem. POMI, \textbf{327} (2005), 98--114; J. Math. Sci., \textbf{139}:2, 6417--6424 (2006)
		
		\bibitem{Lacey} M. Lacey, \emph{    Issues related to Rubio de Francia's Littlewood-Paley inequality}, New York Journal of Mathematics. NYJM Monographs, \textbf{2} (2007)
		
		\bibitem{Kurtz} D. S. Kurtz, \emph{Littlewood--Paley and multiplier theorems on weighted $L^p$ spaces}, Trans. Amer. Math. Soc., \textbf{259}, 235--254 (1980)
		
		\bibitem{L-P} Littlewood, J. E. and Paley, R. E. A. C., \emph{Theorems on {F}ourier series and power series ({II})}, Proc. London Math. Soc. (2), \textbf{42}, 52--89 (1936)
		
		\bibitem{MusSchlag}Camil Muscalu and Wilhelm Schlag, \emph{Classical and multilinear harmonic analysis. Vol. I}, Cambridge Studies in Advanced Mathematics, vol. 137, Cambridge University Press, Cambridge (2013)
		
		\bibitem{Osip} N.~N.~Osipov, \emph{Littlewood--Paley--Rubio de Francia inequality for the Walsh system}, Algebra i Analiz, \textbf{28}:5, 236--246 (2016)
		
		\bibitem{RdF}Jos\'{e} L. Rubio de Francia, \emph{A Littlewood--Paley inequality for arbitrary intervals}, Rev. Mat. Iberoamericana \textbf{1}, 1--14 (1985)
		
		\bibitem{Stein} E.~Stein, \emph{Harmonic analysis: real-variable methods, orthogonality, and oscillatory integrals}, Princeton Mathematical Series, vol. 43, Princeton University Press, Princeton, NJ, 1993, With the
		assistance of Timothy S. Murphy, Monographs in Harmonic Analysis, III. MR 123219
		
		\bibitem{Tsel} A.~S.~Tselishchev, \emph{A Littlewood-Paley-Rubio de Francia inequality for bounded Vilenkin systems}, Sb. Math., \textbf{212}:10, 1491--1502 (2021)
		
		\bibitem{Vil} N. Vilenkin, \emph{On a class of complete orthonormal systems}, Izv. Akad. Nauk SSSR Ser. Mat., \textbf{11}:4, 363--400 (in Russian) (1947)
		
		\bibitem{Wat} Chinami Watari, \emph{On generalized Walsh Fourier series}, Tohoku Math. J. (2), \textbf{10}, 211--241 (1958)
		
		\bibitem{Weisz} Ferenc Weisz, \emph{Martingale Hardy spaces and their applications in Fourier analysis}, Lecture Notes in Mathematics \textbf{1958}, Springer-Verlag, Berlin (1994)
		
		\bibitem{Young_1} Wo-Sang Young, \emph{Mean convergence of generalized {W}alsh-{F}ourier series}, Trans. Amer. Math. Soc., \textbf{218}, 311--320 (1976)
		
		\bibitem{Young_2} Wo-Sang Young, \emph{Almost everywhere convergence of {V}ilenkin-{F}ourier series of {$H^1$} functions}, Proc. Amer. Math. Soc., \textbf{108}:2, 433--441 (1990)
		
		\bibitem{Young_weights} Wo-Sang Young, \emph{Weighted norm inequalities for Vilenkin-Fourier series}, Trans. Amer. Math. Soc., \textbf{340}:1, 273--291 (1993)
		
		\bibitem{Young_3} Wo-Sang Young, \emph{Littlewood--Paley and multiplier theorems for Vilenkin--Fourier series}, Canad. J. Math., \textbf{46}:3, 662--672 (1994)
	\end{thebibliography}
\end{document}